\documentclass[12pt]{elsarticle}
 
\usepackage[margin=1in]{geometry}
\usepackage{natbib}
\usepackage{graphicx}
\usepackage{mathrsfs}
\usepackage{color}
\usepackage{amssymb}
\usepackage{amsfonts}
\usepackage{amsmath}
\usepackage{bbm}
\usepackage{nameref}
\usepackage{tikz}
\usepackage{stackengine}
\usetikzlibrary{bayesnet}
\usepackage{lineno}
\usetikzlibrary{decorations.pathreplacing}
\usepackage{setspace}

\usepackage{stackengine}

\DeclareFontFamily{U}{rcjhbltx}{}
\DeclareFontShape{U}{rcjhbltx}{m}{n}{<->rcjhbltx}{}
\DeclareSymbolFont{hebrewletters}{U}{rcjhbltx}{m}{n}
\newcommand{\norm}[1]{\left\lVert#1\right\rVert}

\let\aleph\relax\let\beth\relax
\let\gimel\relax\let\daleth\relax

\DeclareMathSymbol{\aleph}{\mathord}{hebrewletters}{39}
\DeclareMathSymbol{\beth}{\mathord}{hebrewletters}{98}
\DeclareMathSymbol{\gimel}{\mathord}{hebrewletters}{103}
\DeclareMathSymbol{\daleth}{\mathord}{hebrewletters}{100}
\DeclareMathSymbol{\lamed}{\mathord}{hebrewletters}{108}
\DeclareMathSymbol{\mem}{\mathord}{hebrewletters}{109}
\DeclareMathSymbol{\ayin}{\mathord}{hebrewletters}{96}
\DeclareMathSymbol{\tsadi}{\mathord}{hebrewletters}{118}
\DeclareMathSymbol{\qof}{\mathord}{hebrewletters}{113}
\DeclareMathSymbol{\shin}{\mathord}{hebrewletters}{152}

\usepackage[hidelinks]{hyperref}
\allowdisplaybreaks

\usepackage{subcaption}

\usepackage{algorithm}
\usepackage{algpseudocode}

\newtheorem{theorem}{Theorem}
\usepackage{amsthm}
\newtheorem{prop}{Proposition}
\theoremstyle{definition}
\newtheorem{definition}{Definition}
\newtheorem*{definition*}{Definition}
\newtheorem{lemma}[theorem]{Lemma}

\algtext*{EndFor}
\algtext*{EndIf}
\algtext*{EndFunction}
\algtext*{EndWhile}

\algrenewcommand\algorithmicforall{\textbf{foreach}}
\algrenewcommand\algorithmicindent{.8em}

\usepackage{placeins}

\begin{document}
\journal{}
\begin{frontmatter}

\title{Transit Frequency Setting Problem with Demand Uncertainty}

\author[mitcee]{Xiaotong Guo\corref{cor}}
\ead{xtguo@mit.edu} 
\author[mitcee]{Baichuan Mo}
\ead{baichuan@mit.edu} 
\author[neucee]{Haris N. Koutsopoulos}
\ead{h.koutsopoulos@northeastern.edu}
\author[mitdusp]{Shenhao Wang}
\ead{shenhao@mit.edu}
\author[mitdusp]{Jinhua Zhao\corref{cor}}
\ead{jinhua@mit.edu}

\address[mitcee]{Department of Civil and Environmental Engineering, Massachusetts Institute of Technology, Cambridge, MA 02139, USA}
\address[neucee]{Department of Civil and Environmental Engineering, Northeastern University, Boston, MA 02115}
\address[mitdusp]{Department of Urban Studies and Planning, Massachusetts Institute of Technology, Cambridge, MA 02139, USA}
\cortext[cor]{Corresponding author}

\begin{abstract}
Public transit systems are the backbone of urban mobility systems in the era of urbanization. The design of transit schedules is important for the efficient and sustainable operation of public transit. However, previous studies usually assume fixed demand patterns and ignore uncertainties in demand, which may generate transit schedules that are vulnerable to demand variations. To address demand uncertainty issues inherent in public transit systems, this paper adopts both stochastic programming (SP) and robust optimization (RO) techniques to generate robust transit schedules against demand uncertainty. A nominal (non-robust) optimization model for the transit frequency setting problem (TFSP) under a single transit line setting is first proposed. The model is then extended to SP-based and RO-based formulations to incorporate demand uncertainty. The large-scale origin-destination (OD) matrices for real-world transit problems make the optimization problems hard to solve. To efficiently generate robust transit schedules, a Transit Downsizing (TD) approach is proposed to reduce the dimensionality of the problem. We prove that the optimal objective function of the problem after TD is close to that of the original problem (i.e., the difference is bounded from above). The proposed models are tested with real-world transit lines and data from the Chicago Transit Authority (CTA). Compared to the current transit schedule implemented by CTA, the nominal TFSP model without considering demand uncertainty reduces passengers' wait times while increasing in-vehicle travel times. After incorporating demand uncertainty, both stochastic and robust TFSP models reduce passengers' wait times and in-vehicle travel times simultaneously. The robust TFSP model produces transit schedules with better in-vehicle travel times and worse wait times for passengers compared to the stochastic TFSP model.
\end{abstract}


\end{frontmatter}


\section{Introduction}
\label{sec:intro}

The past century has witnessed one of the most dramatic evolution in human history, urbanization. More than half of the world now lives in urban areas. By 2050, over two-thirds of the world’s population is expected to live in urban areas~\cite{owidurbanization}. Urban mobility, defined as moving people from one place to another within or between urban areas, is critical to the functionality of people's daily lives in urban areas. It allows people to access housing, jobs, and recreational services. However, urban mobility is also the largest contributor to greenhouse gas emissions in the United States, accounting for over 27\% of the total greenhouse gas emissions~\cite{carbon_emission_us}. Therefore, an efficient and sustainable urban mobility system is necessary to support future urban development.

Although emerging urban mobility services, e.g., Ride-hailing and bike-sharing, have provided people with various options for traveling, public transit systems keep serving as the backbone of a sustainable urban mobility system, which allows more efficient travel across cities for a mass number of people. Meanwhile, public transit systems provide a more affordable travel option for everyone regardless of travel distances within cities. Hence, it is important to design a public transit system with a good level of service and operate it efficiently.  

The COVID-19 pandemic has imposed an enormous impact on public transit systems. The national public transportation ridership stays around 60\% of the pre-pandemic ridership level at the beginning of 2022~\cite{national_ridership}. One of the main driving forces for the ridership drop is the flexible or remote working adopted by many employers worldwide during the pandemic. However, remote working won't be a temporary strategy for companies because the US is projected to have an average of 30\% paid full days working from home for people in the future compared to a 5\% pre-pandemic level~\cite{NBERw28731}. Remote working implies that a proportion of commute trips in transit may be lost permanently, which motivates transit agencies to redesign their transit networks and schedules with the new demand patterns. Also, transit demand has become more volatile. Predicting future demand becomes more challenging.

While transit networks have been developed for years and are hard to change by transit agencies within a short period of time, changing transit schedules is straightforward. In this paper, we focus on the transit frequency setting problem (TFSP), where transit schedules are optimized given a set of transit stops to serve. Though TFSP has been explored in previous literature, there is a limited number of papers incorporating uncertainty (such as volatile demand) into consideration for the TFSP~\cite{Ibarra_2015}. Ignoring demand uncertainty when setting up transit schedules may diminish the level of service for transit systems.  

To handle demand uncertainty for transit systems, especially during the post-COVID era, we first propose a baseline TFSP model for a single transit line. Next, we introduce two techniques to incorporate uncertainty into decision-making processes: Stochastic Programming (SP) and Robust Optimization (RO). Furthermore, the Transit Downsizing (TD) approach is proposed to reduce problem dimensionality and generate optimal transit schedules efficiently. Also, the proposed TD approach can be utilized in other transit-related problems. Overall, the contribution of this paper can be summarized as follows:

\begin{itemize}
    \item Propose a nominal TFSP model for a single transit line and an extended TFSP model considering crowding levels.
    \item Address demand uncertainty issues by introducing a stochastic TFSP model considering multiple demand scenarios and a robust TFSP model which optimizes for the worst-case demand scenario.
    \item Design the TD approach to reduce problem sizes and make the model tractable given large-scale demand matrices from real-world transit instances. Theoretically prove that the optimal objective function of the problem after TD is close to that of the original problem (i.e., the difference is bounded from above)
    \item Compare the current transit schedule with the schedules solved by nominal, stochastic, and robust optimization, respectively, under multiple demand scenarios over two real-world transit lines (Routes 49 and X49) operated in Chicago. Both stochastic and robust optimization-based methods reduce passenger wait times and in-vehicle travel times simultaneously compared to the status quo. Compared to stochastic optimization, robust optimization reduces in-vehicle travel times while increasing wait times for passengers.
\end{itemize}

The remainder of the paper is organized as follows. 
Section \ref{sec:literature_review} reviews the relevant literature. 
Section \ref{sec:methodology} describes the nominal, stochastic and robust TFSP models and proposed dimensionality reduction algorithms.
Section \ref{sec:results} outlines experimental setups, including utilized data and transit lines, and displays experiment results and sensitivity analyses. 
Finally, Section \ref{sec:conclusion} recaps the main contributions of this work, outlines the limitations, and provides future research directions.

\section{Literature Review}
\label{sec:literature_review}

\subsection{Transit Frequency Setting Problem}

The design and planning of urban public transit systems consist of a series of decisions before operating the system, which is known as Transit Network Planning (TNP) problem. 
In literature, TNP is commonly divided into sub-problems that range across tactical, strategical, and operational decisions, including Transit Network Design (TND), Frequency Setting (FS), Transit Network Timetabling (TNT), Vehicle Scheduling Problem (VSP), Driver Scheduling Problem (DSP), Driver Rostering Problem (DRP). A thorough review of TNP and its sub-problems can be found in \cite{Ibarra_2015, Farahani2013, Ceder2007}.

The TFSP is defined as a problem to determine the number of trips for a given set of lines that provide a high level of service in a planning period. The TFSP is first studied by \citet{Newell1971} using analytic models. Given a fixed number of vehicles and constant passenger arrival rate, \citet{Newell1971} produced vehicle dispatching time in order to minimize the total waiting time of all passengers. He concluded that the optimal headway should be approximate as the square root of the arrival rate of passengers. His proposed model assumes fixed passenger demand and overlooks vehicle capacity constraints. 

\citet{Furth1981} formulated the TFSP as a non-linear program that computed the optimal headway for bus routes in order to maximize the net social benefits, consisting of ridership benefits and wait-time savings. Sets of constraints incorporated in their model were total subsidy, maximum fleet size, and acceptable level of loading. A key assumption they made was considering responsive demand which was a function of headway in the model. Furthermore, a heuristic-based algorithm was designed to solve non-linear programs.

More recently, \citet{Verbas2013} extended the model proposed by \citet{Furth1981} by incorporating service patterns into transit routes. A service pattern corresponds to a unique set of steps that need to be served by transit vehicles along a transit route. They formulated two non-linear optimization problems with different objectives: i) maximize the number of riders and wait time savings, and ii) minimize the net cost. Non-linear optimization solvers were directly used to solve non-linear programs. Additionally, \citet{Verbas2015} discussed the impact of demand elasticity over solutions from the TFSP which is similar to models proposed by \citet{Furth1981} and \citet{Verbas2013}. They introduced three methodologies for estimating demand elasticity within transit networks and solved TFSPs under multiple demand elasticity scenarios on a large-scale network. Although the impact of demand uncertainty is discussed in this paper, their proposed methods are not equipped with abilities to generate optimal schedules considering demand uncertainty explicitly.

One could argue that one of the modeling contributions in formulations based on \citet{Furth1981}'s model is the introduction of responsive demand. However, the authors claim that it is more reasonable to consider a fixed demand matrix when solving the TFSP. There are short-term and long-term objectives in the TFSP: i) minimizing wait times for existing passengers, and ii) attracting more passengers to use transit networks. Minimizing wait times for the existing passengers leads to an increase in the level of service, which in turn attracts more passengers to take transit. On the contrary, maximizing ridership when considering responsive demand could lead to a waste of resources since it takes weeks for demand to respond to service changes. Meanwhile, transit schedules are modified frequently in practice, e.g., Chicago Transit Authority (CTA) publishes new transit schedules quarterly. An updated demand matrix can be utilized when generating new transit schedules every time. Therefore, minimizing wait times for existing passengers is a better objective in the authors' opinion.

Although limited papers take demand uncertainty into consideration when setting transit frequencies, \citet{Li2013} 
utilized stochastic programming techniques to solve the headway optimization problem for a single bus route considering random passenger arrivals, boarding, alighting, and vehicle travel time. A metaheuristic algorithm consisting of a stochastic simulation and a genetic algorithm was designed to solve the proposed model. Their proposed approach was compared with three traditional headway determination models and bringing both demand and travel time uncertainty improved model performances. The main critique for \citet{Li2013}'s work is the lack of discussions on the optimality gap given a heuristic-based solution algorithm. In this paper, we will propose a stochastic TFSP model with a dimensionality reduction approach that can be solved optimally and efficiently. 


\subsection{Stochastic Programming, Robust Optimization, and Applications in Urban Mobility}

There are two widely-used approaches for decision-making under uncertainty in the Operations Research (OR) domain: SP~\cite{Birge_Louveaux_2011} and RO~\cite{Ben-Tal_2009}. For the SP approach, the most traditional method is Sample Average Approximation (SAA), where the true distributions over uncertain parameters are approximated by empirical distributions obtained from the data~\cite{Kleywegt_Shapiro_Homem-de-Mello_2002}. On the other hand, RO and its data-driven variants~\cite{Bertsimas2018} is another option to handle uncertain parameters effectively. The underlying idea for RO is to specify a range for an uncertain parameter, namely an uncertainty set, and optimize over the worst-case realizations given the bounded uncertainty set. The solution method for RO problems involves generating a deterministic equivalent, called the robust counterpart. A practical guide on RO can be found in \cite{Gorissen_2015}.

Urban mobility systems have various sources of uncertainty brought by human behaviors and environmental impacts (e.g., weather). Considering uncertainty when designing and operating urban mobility systems is crucial and necessary. There are several applications for applying SP or RO techniques to solve urban mobility problems. For transit systems, \citet{Yan2013} proposed a robust framework for solving the bus transit network design problem considering stochastic travel times. \citet{Mo2022} utilized the RO technique to solve the individual path recommendation problem under rail disruptions considering demand uncertainty. For shared mobility systems, \citet{GUO2021161} formulated a robust matching-integrated vehicle rebalancing (MIVR) model to balance vacant vehicles in the ride-hailing operations given demand uncertainty. \citet{GUO2022} extended the MIVR model proposed by \citet{GUO2021161} by introducing predictive prescriptions approach~\cite{Dimitris2020} to handle demand uncertainty, which is an advanced approach for handling data uncertainty based on the stochastic optimization framework.

\section{Methodology}
\label{sec:methodology}

\subsection{Basic Optimization Model}

We consider the TFSP for a single urban transit line (either rail or bus services) with a sequence of $N$ stops. Let the set of stops be $\mathcal{S}$. A single line is the basic element of a transit network. Future studies can be extended to the network-level design by considering potential interactions between different lines.
Without loss of generality, we assume each bi-directional transit line is considered as two separate transit lines with distinct sets of stops in this paper.
For an urban transit line, there exists a set of potential service patterns $\mathcal{P}$, where each pattern $p \in \mathcal{P}$ consists of a subset of stops $\mathcal{S}_p \subseteq \mathcal{S}$, indicating where the vehicles should stop if traveling with this pattern. Common examples of patterns are short-turnings and limited-stop lines in bus operations.

Let $\mathcal{V}$ represent the set of vehicle types that can be operated on the transit line. For instance, $\mathcal{V} = \{\text{standard bus},\; \text{articulated bus}, \; \text{minibus}\}$ includes three types of buses, and $\mathcal{V} = \{\text{four-car train},\; \text{six-car train}, \; \text{eight-car train}\}$ consists of three types of rail cars with a different number of carriages.
For each type of vehicle $v \in \mathcal{V}$, the number of seats is $C_v$ and the maximum vehicle capacity is $\Bar{C}_v$. 
Furthermore, we discretize the full planning period $[T_{start}, T_{end}]$ into time periods $t = 1,...,T$, where each time interval $t$ has the same length $\Delta$. 

Let \emph{passenger flow} $(o, d, t)$ stand for passengers with origin station (stop) $o \in \mathcal{S}$ and destination station (stop) $d \in \mathcal{S}$ who arrives at the boarding station (stop) $o$ at the beginning of time interval $t$.
The set of passenger flows is indicated by $\mathcal{F}$.
For each transit line, we have a demand matrix $\boldsymbol{u} = (u_t^{o,d})$, where $u_t^{o, d}$ indicates demand for the passengers flow $(o, d, t)$.
The decision variables for the TFSP are $\boldsymbol{x} = (x_t^{p, v})$, where $x_t^{p, v} = 1$ denotes a vehicle with type $v \in \mathcal{V}$ operating on a pattern $p \in \mathcal{P}$ departures from the terminal station of pattern $p$ at the beginning of time interval $t$. Hence, unlike typical headway-based design, this paper allows non-even dispatching of vehicles according to the service needs.

In real-world transit line operations, transit agencies usually have a limited number of operating patterns for each line due to practical constraints. 
Therefore, we impose a \emph{sparsity} constraint on operating patterns. Define an auxiliary decision variable $y_p, \forall p \in \mathcal{P}$, where $y_p=1$ indicates that the pattern $p$ can be operated on the transit line. 
Let $P$ represent the maximum number of patterns operated on a single transit line. The sparsity constraint can be formulated as 
\begin{subequations}
\label{eq:sparsity}
\begin{align}
    & x_t^{p,v} \leq y_p \quad \forall t=1,...,T, \forall p \in \mathcal{P}, \forall v \in \mathcal{V}, \\
    & \sum_{p \in \mathcal{P}} y_p \leq P.
\end{align}
\end{subequations}

Let $c^{p, v}$ stand for the cost parameter associated with operating a vehicle of type $v$ on a pattern $p$.
The budget for scheduling transit services over the transit line is represented by $B$.
The set of feasible schedules is denoted by 
\begin{equation}
\label{eq:feasible_x}
\small
    \mathcal{X}_B = \{ \boldsymbol{x} \in \{0, 1\}^{|\mathcal{P}| \times |\mathcal{V}| \times T}: \sum_{p \in \mathcal{P}} \sum_{v \in \mathcal{V}} \sum_{t = 1}^T c^{p, v} x_t^{p, v} \leq B; \; \sum_{v \in \mathcal{V}} x_t^{p, v} \leq 1, \; \forall t = 1,...,T, \forall p \in \mathcal{P};  \text{Constraints} (\ref{eq:sparsity})\}.
\end{equation}

The feasibility constraints in Equation (\ref{eq:feasible_x}) ensure that the total scheduled transit services do not exceed the budget $B$ and only one type of vehicle can be operated on each pattern during each time interval $t$\footnote{It is worth mentioning that multiple patterns are allowed to be operated within the same time period.}.
Equation (\ref{eq:feasible_x}) imposes a general budget constraint, which can be modified to incorporate more complicated cases. 
For instance, the budget constraint can be adapted to ensure a limited number of vehicles for each vehicle type:
$$\sum_{p \in \mathcal{P}} \sum_{t=1}^T x_t^{p, v} \leq B_v \quad \forall v \in \mathcal{V},$$
where $B_v$ is the number of available vehicles for each vehicle type $v$ and the cost parameter $c^{p, v} = 1, \forall p \in \mathcal{P}, \forall v \in \mathcal{V}$. Meanwhile, additional constraints can be added to incorporate agency-specific constraints. For example, $\sum_{v \in \mathcal{V}} x_t^{p, v} \geq 5$ implies that at least 5 buses need to be scheduled to operate with pattern $p$ during time $t$. 

For each passenger flow $(o, d, t)$, let $\mathcal{P}^{o,d} \subseteq \mathcal{P}$ denote the set of patterns that includes both stations $o$ and $d$. 
To capture boarding for passenger flows, we define decision variables $\boldsymbol{\lambda} = (\lambda_{t, \tau}^{o, d, p, v})$, where $\lambda_{t, \tau}^{o, d, p, v} \in \mathbb{R}_+$\footnote{We relax the integer variable $\boldsymbol{\lambda}$ to continuous variable to increase tractability for solving the problem while maintaining a satisfying model performance.} indicates the number of passengers in the passenger flow $(o, d, t)$ who board on a vehicle $v$ that departs at the first station of pattern $p$ at time $\tau$.

Let $w_{t, \tau}^{o, d, p, v}$ represent the waiting time for the passenger flow $(o, d, t)$ to board the vehicle $v$ which departs at the first station of the pattern $p$ at time $\tau$.
For passengers with an origin-destination pair $(o, d)$, the in-vehicle travel time for them to take a transit vehicle operating on pattern $p$ is represented as $\phi^{o,d,p}$.
Note that to maintain the linearity of the optimization model, we assume a pattern-specific fixed travel time\footnote{Can be extended to time-dependent travel time.}. Dwell times are also ignored in the model since they are generally small compared to in-vehicle times. Let $L_{\tau}^{p, v, s}$ stands for the vehicle load after visiting the station $s \in \mathcal{S}_p$ of vehicle $v$ which departures at the first station of the pattern $p$ at time $\tau$, i.e.,
\begin{equation}
\small
\label{eq:load_function}
    L_{\tau}^{p, v, s} = \sum_{o \in \mathcal{S}_p^{\text{before}}(s)} \sum_{d \in \mathcal{S}_p^{\text{after}}(s)} \sum_{t=1}^{T_{\tau,p}^{o,d}} \lambda_{t, \tau}^{o, d, p, v} \quad \forall p \in \mathcal{P}, \forall v \in \mathcal{V}, \forall s \in \mathcal{S}_p, \forall \tau = 1,...,T,
\end{equation}
where $\mathcal{S}_p^{\text{before}}(s), \mathcal{S}_p^{\text{after}}(s)$ indicate sets of stations in $\mathcal{S}_p$ which are before (include station $s$) and after the station $s$, respectively.
$T_{\tau,p}^{o,d}$ indicates the latest time interval such that a passenger with the origin-destination pair $(o,d)$ can board a transit vehicle that departs from the first station at the time $\tau$ with pattern $p$.

To guarantee the feasibility of the model, we introduce an auxiliary decision variable $\boldsymbol{\eta} = (\eta_{t}^{o,d} \geq 0)$ indicating the number of unsatisfied passenger flow $(o, d, t)$ (i.e., passengers who can not be served by the transit system). $\boldsymbol{\eta}$ serves as a slack variable to guarantee the problem always has feasible solutions. Hence, the flow conservation constraints can be represented as:
\begin{align} 
    \sum_{v \in \mathcal{V}} \sum_{p \in \mathcal{P}^{o, d}} \sum_{\tau=\tau_t^{o,d,p}}^T \lambda_{t, \tau}^{o, d, p, v} = u_t^{o,d} - \eta_t^{o,d} \quad \forall (o, d, t) \in \mathcal{F} \label{eq:flow_conservation}
\end{align}
where $\tau_t^{o, d, p}$ represents the earliest departure time for vehicles that are operated on a pattern $p \in \mathcal{P}^{o,d}$ and can be boarded by the passenger flow $(o, d, t)$. Eq. (\ref{eq:flow_conservation}) means that all passengers from a passenger flow will board vehicles or stay unsatisfied.

Then, we have the following Integer Linear Programming (ILP) formulation for setting optimal frequencies for urban transit lines:
\begin{subequations}
\small
\label{eq:basic_ILP}
\begin{align}
    (P) \quad \min_{\boldsymbol{x} \in \mathcal{X}_B, \boldsymbol{\lambda}, \boldsymbol{\eta}} \quad & \sum_{(o, d, t) \in \mathcal{F}} \sum_{v \in \mathcal{V}} \sum_{p \in \mathcal{P}^{o, d}} \sum_{\tau=\tau_t^{o,d,p}}^T  \left( w_{t, \tau}^{o, d, p, v} + \gamma \phi^{o,d,p} \right)  \lambda_{t, \tau}^{o, d, p, v} + M \sum_{(o, d, t) \in \mathcal{F}} \eta_{t}^{o,d} \\
    \text{s.t.} \quad
    & \text{Constraints} \; (\ref{eq:load_function}) \text{ and } (\ref{eq:flow_conservation})\notag \\
    & L_{\tau}^{p, v, s} \leq \Bar{C}_v  x_\tau^{p, v}  \quad \forall p \in \mathcal{P}, \forall v \in \mathcal{V}, \forall s \in \mathcal{S}_p, \forall \tau = 1,...,T;\\
    & \lambda_{t, \tau}^{o, d, p, v} \geq 0 \quad \forall (o, d, t) \in \mathcal{F}, \forall p \in \mathcal{P}, \forall v \in \mathcal{V}, \forall \tau = 1,...,T;\\
    & \eta_{t}^{o,d} \geq 0 \quad \forall (o,d,t) \in \mathcal{F}.
\end{align}
\end{subequations}

The objective function (\ref{eq:basic_ILP}a) minimizes the total generalized journey time for passengers who take transit services and the penalty of unsatisfied passenger flows.
$\gamma$ is a weight parameter controlling the importance between wait times and in-vehicle travel times. $\gamma = 0$ leads to a problem that only minimizes passengers' wait times and $\gamma = 1$ generates a problem that minimizes passengers' journey times (i.e., wait plus in-vehicle times). $M$ stands for a large number that dominates the objective function (\ref{eq:basic_ILP}a), indicating that all passenger flows should be served in the transit system. 
Constraints (\ref{eq:basic_ILP}b) guarantee that passenger loads on vehicles do not exceed the vehicle capacity.
Constraints (\ref{eq:basic_ILP}c) and (\ref{eq:basic_ILP}d) ensure that decision variables $\boldsymbol{\lambda}$ and $\boldsymbol{\eta}$ are non-negative.

\subsection{Optimization Model with Crowding Extension}
Passengers may have different comfort levels depending on the degree of crowding in a vehicle and whether they can have a seat or not. Also, the potential infection risks of COVID-19 require transit agencies to control the vehicle load. 
To grant the model the ability to control crowding levels on transit vehicles, we introduce a binary decision variable $\boldsymbol{z} = (z_t^{p,v,s})$, where $z_t^{p,v,s} = 1$ indicates the vehicle with type $v$ operating on a pattern $p$ and departing from the terminal at time $t$ is \emph{crowded} at the segment (consecutive station to station pair) after passing through station $s$. A transit vehicle $v$ is $\emph{crowded}$ if the passenger load on the vehicle is greater than the seated capacity $C_v$\footnote{When passenger loading exceeds seated capacity, the proportion of passengers must stand and standees perceive up to 2.25 times actual travel time~\cite{TRB_TCQS}.}. 

Let $\omega$ represent the penalty cost per unit of travel time of a crowded transit vehicle. For a vehicle operating on a pattern $p$, let $\phi^{p,s}$ denote the vehicle running time of the segment after passing through station $s$. The ILP with crowding extension can be formulated as follows:
\begin{subequations}
\small
\label{eq:crowding_ILP}
\begin{align}
    (P-C) \quad \min_{\boldsymbol{x} \in \mathcal{X}_B, \boldsymbol{\lambda}, \boldsymbol{\eta}, \boldsymbol{z}} \quad & \sum_{(o, d, t) \in \mathcal{F}} \sum_{v \in \mathcal{V}} \sum_{p \in \mathcal{P}^{o, d}} \sum_{\tau=\tau_t^{o,d,p}}^T  \left( w_{t, \tau}^{o, d, p, v} + \gamma \phi^{o,d,p} \right)  \lambda_{t, \tau}^{o, d, p, v} + M \sum_{(o, d, t) \in \mathcal{F}} \eta_{t}^{o,d} \notag \\ 
    & + \omega \sum_{p \in \mathcal{P}} \sum_{v \in \mathcal{V}} \sum_{s\in\mathcal{S}_p} \sum_{\tau=1}^{T} z_\tau^{p,v,s} \phi^{p,s}\\
    \text{s.t.} \quad
    & \text{Constraints} \; (\ref{eq:load_function}), (\ref{eq:flow_conservation}), (\ref{eq:basic_ILP}c), (\ref{eq:basic_ILP}d) \notag \\
    & L_{\tau}^{p, v, s} \leq C_v x_\tau^{p, v} + (\Bar{C}_v - C_v) z_\tau^{p, v,s} \quad \forall p \in \mathcal{P}, \forall v \in \mathcal{V}, \forall s \in \mathcal{S}_p, \forall \tau = 1,...,T;\\
    & z_{\tau}^{p,v,s} \leq x_{\tau}^{p,v} \quad \forall p \in \mathcal{P}, \forall v \in \mathcal{V}, \forall s \in \mathcal{S}_p, \forall \tau = 1,...,T; \\
    & z_\tau^{p,v,s} \in \{0, 1\} \quad \forall p \in \mathcal{P}, \forall v \in \mathcal{V}, \forall s \in \mathcal{S}_p, \forall \tau = 1,...,T.
\end{align}
\end{subequations}

Besides the objective for problem $(P)$, the crowding penalty for transit vehicles is also added to the objective function as (\ref{eq:crowding_ILP}a), which leads to a transit schedule and passenger boarding choices minimizing the crowding levels. 
When $\omega = 0$, the problem (\ref{eq:crowding_ILP}) is equivalent to the problem (\ref{eq:basic_ILP}), leading to transit schedules that minimize the total generalized journey time for passengers given passengers will board the first available transit vehicles.
When $\omega > 0$, we assume passengers can wait for the next transit vehicle in order to reduce the crowding levels. It is worth noting that, in reality, passengers may or may not board a crowded vehicle depending on their comfort level requirement \citep{mo2020capacity}. Our model simplifies the modeling of passengers' willingness to board and assumes that their boarding behavior minimizes the objective function. Hence, the objective function is a lower bound of the actual system cost. In this way, our model is useful for providing a perspective of system optimum and showing the trade-off between passengers' total waiting time and crowding levels in transit vehicles.
Constraints (\ref{eq:crowding_ILP}b) are the modified capacity constraints with crowding level. Constraints (\ref{eq:crowding_ILP}c) restrict that a vehicle can only be crowded if it is operated in the system. Constraints (\ref{eq:crowding_ILP}d) specify decision variable $\boldsymbol{z}$ is binary.

\subsection{Stochastic Programming Model Formulation}
\label{sybsec:stoc_prog}

The demand matrix $u_t^{o, d}$ in the problem $(P)$ is critical for the TFSP. However, existing literature usually assumes a constant demand matrix $\boldsymbol{u}$ estimated from the historical data. 
In this paper, we first introduce a stochastic TFSP model extended from the nominal TFSP model $(P)$ to design transit schedules considering multiple demand scenarios.

Given a set of demand scenarios $\mathcal{E}$, the corresponding demand matrix $\boldsymbol{u}_e$ for a demand scenario $e \in \mathcal{E}$ has probability $p_e$. By introducing demand scenarios into the frequency setting problem, we adjust the boarding decision variables for passengers to $\boldsymbol{\lambda}_e = (\lambda_{t, \tau, e}^{o,d,p,v})$ for each demand scenario $e \in \mathcal{E}$, where $\lambda_{t, \tau, e}^{o,d,p,v} \in \mathbb{R}_+$ represents the number of passengers in the passenger flow $(o, d, t)$ who board on a vehicle $v$ which departures at the beginning of pattern $p$ at time $\tau$ under demand scenario $e$. 
Similarly, auxiliary variables $\boldsymbol{\eta}$ are extended to $\boldsymbol{\eta}_e = (\eta_{t,e}^{o,d})$ for each demand scenario $e \in \mathcal{E}$. 
Then the stochastic TFSP model can be formulated as:
\begin{subequations}
\small
\label{eq:ILP_SP}
\begin{align}
    (SP) \quad \min_{\boldsymbol{x} \in \mathcal{X}_B, \boldsymbol{\lambda}, \boldsymbol{\eta}} \quad & \sum_{e \in \mathcal{E}} p_e \left[ \sum_{(o, d, t) \in \mathcal{F}} \sum_{v \in \mathcal{V}} \sum_{p \in \mathcal{P}^{o, d}} \sum_{\tau=\tau_t^{o,d,p}}^T  \left( w_{t, \tau}^{o, d, p, v} + \gamma \phi^{o,d,p} \right)  \lambda_{t, \tau, e}^{o, d, p, v} + M \sum_{(o, d, t) \in \mathcal{F}} \eta_{t,e}^{o,d} \right] \\
    \text{s.t.} \quad
    & L_{\tau, e}^{p, v, s} = \sum_{o \in \mathcal{S}_p^{\text{before}}(s)} \sum_{d \in \mathcal{S}_p^{\text{after}}(s)} \sum_{t=1}^{T_{\tau,p}^{o,d}} \lambda_{t, \tau, e}^{o, d, p, v} \quad \forall p \in \mathcal{P}, \forall v \in \mathcal{V}, \forall s \in \mathcal{S}_p, \forall \tau = 1,...,T, \forall e \in \mathcal{E}; \\
    & \sum_{v \in \mathcal{V}} \sum_{p \in \mathcal{P}^{o, d}} \sum_{\tau=\tau_t^{o,d,p}}^T \lambda_{t, \tau, e}^{o, d, p, v} = u_{t,e}^{o,d} - \eta_{t,e}^{o,d} \quad \forall (o, d, t) \in \mathcal{F}, \forall e \in \mathcal{E}; \\
    & L_{\tau, e}^{p, v, s} \leq \Bar{C}_v  x_\tau^{p, v}  \quad \forall p \in \mathcal{P}, \forall v \in \mathcal{V}, \forall s \in \mathcal{S}_p, \forall \tau = 1,...,T, \forall e \in \mathcal{E};\\
    & \lambda_{t, \tau, e}^{o, d, p, v} \geq 0 \quad \forall (o, d, t) \in \mathcal{F}, \forall p \in \mathcal{P}, \forall v \in \mathcal{V}, \forall \tau = 1,...,T, \forall e \in \mathcal{E};\\
    & \eta_{t, e}^{o,d} \geq 0 \quad \forall (o,d,t) \in \mathcal{F}, \forall e \in \mathcal{E}.
\end{align}
\end{subequations}

The problem $(SP)$ is a stochastic extension of the nominal optimization problem $(P)$, and we minimize the expected total generalized journey time and penalties induced by unsatisfied demand across all demand scenarios. The number of variables and constraints grows linearly regarding the number of demand scenarios $|\mathcal{E}|$.

\subsection{Robust Optimization Model Formulation}
\label{subsec:robust_opt}

Besides using SP to handle demand uncertainty when setting transit frequencies, RO~\cite{Ben-Tal_2009} is another approach widely used in literature for decision-making under uncertainty. Compared to SP where the generated transit schedules are optimal for an ``average'' demand scenario, RO produces transit schedules that are optimized against the worst-case demand scenario. The motivation for introducing RO into transit frequency setting is that transit operators would prefer no passengers suffer from excessive wait times given any demand scenarios.

To construct a robust TFSP model, we define an uncertainty set around the uncertain demand parameter $u_t^{o, d}$. The uncertainty set specifies a range for the uncertain demand $u_t^{o, d}$ where $u_t^{o, d}$ can change to any level within the range. Transit schedules are then generated using RO techniques with respect to the worst-case demand scenario in the uncertainty set. 

We adopted the budget uncertainty set introduced by \citet{bertsimas2004price}, which is widely used in literature, to quantify the demand uncertainty in the TFSP. Let $\mu_t^{o, d}, \sigma_t^{o, d}$ denote the mean and standard deviation of the demand of passenger flow $(o, d, t)$ derived from the historical data, respectively. The budget uncertainty set is defined as
\begin{equation}
\label{eq:budget_set}
    \mathcal{U}(\Gamma) = \left\{ \boldsymbol{u}: \left| \frac{u_t^{o,d} - \mu_t^{o, d}}{\sigma_t^{o, d}} \right| \leq 1, \forall (o,d,t) \in \mathcal{F}; \sum_{(o,d,t) \in \mathcal{F}} \left| \frac{u_t^{o,d} - \mu_t^{o, d}}{\sigma_t^{o, d}} \right| \leq \Gamma \right\},
\end{equation}

where $\Gamma$ is a parameter controlling the level of uncertainty for the budget uncertainty set. 
The budget uncertainty set implies that the demand can deviate from its historical average by at most one standard deviation, and the total absolute deviations for all passenger flows is upper-bounded by $\Gamma$. 
Define an uncertain parameter $\boldsymbol{\zeta} \in \mathbb{R}^{|\mathcal{F}|}$ and let $u_t^{o, d} = \mu_t^{o, d} + \sigma_t^{o, d} \zeta_t^{o, d}$. We have the following reformulated uncertainty set:
\begin{equation}
\label{eq:budget_set}
    \mathcal{U}(\Gamma) = \left\{ \boldsymbol{\zeta}: \norm{\boldsymbol{\zeta}}_{\infty} \leq 1, \norm{\boldsymbol{\zeta}}_{1} \leq \Gamma \right\}.
\end{equation}

With the defined uncertainty set over demand vector $\boldsymbol{u}$, we propose the robust TFSP model:
\begin{subequations}
\small
\label{eq:robust_model}
\begin{align}
    (RO) \quad \min_{\boldsymbol{x} \in \mathcal{X}_B, \boldsymbol{\lambda}, \boldsymbol{\eta}} \quad & \sum_{(o, d, t) \in \mathcal{F}} \sum_{v \in \mathcal{V}} \sum_{p \in \mathcal{P}^{o, d}} \sum_{\tau=\tau_t^{o,d,p}}^T \left( w_{t, \tau}^{o, d, p, v} + \gamma \phi^{o, d, p} \right) \lambda_{t, \tau}^{o, d, p, v} + M \sum_{(o, d, t) \in \mathcal{F}} \eta_{t}^{o,d} \\
    \text{s.t.} \quad
    & \sum_{v \in \mathcal{V}} \sum_{p \in \mathcal{P}^{o, d}} \sum_{\tau=\tau_t^{o,d,p}}^T \lambda_{t, \tau}^{o, d, p, v} = \mu_t^{o, d} + \sigma_t^{o, d} \zeta_t^{o, d} - \eta_{t}^{o,d} \quad \forall (o, d, t) \in \mathcal{F}, \forall \boldsymbol{\zeta} \in \mathcal{U}(\Gamma); \\
    & \sum_{o \in \mathcal{S}_p^{\text{before}}(s)} \sum_{d \in \mathcal{S}_p^{\text{after}}(s)} \sum_{t=1}^{T_{\tau,p}^{o,d}} \lambda_{t, \tau}^{o, d, p, v} \leq \Bar{C}_v x_\tau^{p, v} \quad \forall p \in \mathcal{P}, \forall v \in \mathcal{V}, \forall s \in \mathcal{S}_p, \forall \tau = 1,...,T; \\
    & \lambda_{t, \tau}^{o, d, p, v} \geq 0 \quad \forall (o, d, t) \in \mathcal{F}, \forall p \in \mathcal{P}, \forall v \in \mathcal{V}, \forall \tau = 1,...,T; \\
    & \eta_t^{o,d} \geq 0 \quad \forall (o,d,t) \in \mathcal{F}.
\end{align}
\end{subequations}

Constraints (\ref{eq:robust_model}b) in problem $(RO)$ are equality constraints with uncertain parameters which often restrict the feasibility region drastically or even lead to infeasibility~\cite{Gorissen_2015}. Therefore, we eliminate variables $\eta_{t}^{o,d}$ via substitution. Equality constraints (\ref{eq:robust_model}b) can be reformulated as
\begin{equation}
\label{eq:reformulation_eta}
\small
    \eta_{t}^{o,d} = \mu_t^{o, d} + \sigma_t^{o, d} \zeta_t^{o, d} - \sum_{v \in \mathcal{V}} \sum_{p \in \mathcal{P}^{o, d}} \sum_{\tau=\tau_t^{o,d,p}}^T \lambda_{t, \tau}^{o, d, p, v} \quad \forall (o, d, t) \in \mathcal{F}, \forall \boldsymbol{\zeta} \in \mathcal{U}(\Gamma).
\end{equation}

Substituting Constraints (\ref{eq:reformulation_eta}) into the objective function (\ref{eq:robust_model}a) and introducing a dummy variable $\alpha$ transform problem $(RO)$ into a problem formulation without equality constraints:
\begin{subequations}
\small
\label{eq:robust_model_2}
\begin{align}
    (RO') \quad \min_{\boldsymbol{x} \in \mathcal{X}_B, \boldsymbol{\lambda}} \quad & \alpha \\
    \text{s.t.} \quad
    & \sum_{(o, d, t) \in \mathcal{F}} \sum_{v \in \mathcal{V}} \sum_{p \in \mathcal{P}^{o, d}} \sum_{\tau=\tau_t^{o,d,p}}^T \left( w_{t, \tau}^{o, d, p, v} + \gamma \phi^{o, d, p} \right) \lambda_{t, \tau}^{o, d, p, v} + M \sum_{(o, d, t) \in \mathcal{F}} \left( \mu_t^{o, d} + \sigma_t^{o, d} \zeta_t^{o, d} \right) \notag \\ 
    & - M \sum_{(o, d, t) \in \mathcal{F}} \sum_{v \in \mathcal{V}} \sum_{p \in \mathcal{P}^{o, d}} \sum_{\tau=\tau_t^{o,d,p}}^T \lambda_{t, \tau}^{o, d, p, v} \leq  \alpha \quad \forall \boldsymbol{\zeta} \in \mathcal{U}(\Gamma); \\
    & \sum_{o \in \mathcal{S}_p^{\text{before}}(s)} \sum_{d \in \mathcal{S}_p^{\text{after}}(s)} \sum_{t=1}^{T_{\tau,p}^{o,d}} \lambda_{t, \tau}^{o, d, p, v} \leq \Bar{C}_v x_\tau^{p, v} \quad \forall p \in \mathcal{P}, \forall v \in \mathcal{V}, \forall s \in \mathcal{S}_p, \forall \tau = 1,...,T; \\
    & \mu_t^{o, d} + \sigma_t^{o, d} \zeta_t^{o, d} - \sum_{v \in \mathcal{V}} \sum_{p \in \mathcal{P}^{o, d}} \sum_{\tau=\tau_t^{o,d,p}}^T \lambda_{t, \tau}^{o, d, p, v} \geq 0 \quad \forall (o, d, t) \in \mathcal{F}, \forall \boldsymbol{\zeta} \in \mathcal{U}(\Gamma); \\
    & \lambda_{t, \tau}^{o, d, p, v} \geq 0 \quad \forall (o, d, t) \in \mathcal{F}, \forall p \in \mathcal{P}, \forall v \in \mathcal{V}, \forall \tau = 1,...,T.
\end{align}
\end{subequations}

However, equivalent formulations do not necessarily lead to equivalent robust counterparts, which are solvable reformulations of robust optimization problems. To guarantee an identical robust counterpart, the substituted variable $\boldsymbol{\eta}$ needs to be \emph{adaptive}, meaning that $\boldsymbol{\eta}(\boldsymbol{\zeta})$ becomes a function of uncertain parameter $\boldsymbol{\zeta}$. Linear Decision Rules (LDRs) are a commonly-used approximation method in literature to handle adaptive robust optimization problems~\cite{Ben-Tal_2009, Bertsimas2020}, which achieve satisfying performances in practice. \citet{Gorissen_2015} suggests that making uncertain variables adaptive and applying LDRs is equivalent to eliminating these variables, given coefficients of such variables do not include uncertain parameters and equality constraints are linear in uncertain parameters. Therefore, our reformulated robust optimization problem $(RO')$ is an approximated formulation of the original robust formulation $(RO)$, which are more tractable to solve without equality constraints.
 
To derive the robust counterpart of problem $(RO')$, which is a solvable formulation of the robust model, the following lemma is introduced~\cite{Bertsimas2020}.
\begin{lemma}
\label{lemma:budget_RC}
    For a constraint $$(\Bar{\boldsymbol{a}} + \boldsymbol{P}\boldsymbol{z})^T\boldsymbol{ x} \leq b, \quad \forall \boldsymbol{z}:\norm{\boldsymbol{z}}_{\infty}\leq \rho, \norm{\boldsymbol{z}}_{1} \leq \Gamma,$$
    it is satisfied by $\boldsymbol{x}$ if and only if there exists an auxiliary variable $\boldsymbol{y}$ such that $(\boldsymbol{x}, \boldsymbol{y})$ satisfies
    $$\Bar{\boldsymbol{a}}^T + \rho \norm{\boldsymbol{y}}_1 + \Gamma \norm{\boldsymbol{P}^T\boldsymbol{x} - \boldsymbol{y}}_{\infty} \leq b.$$
\end{lemma}

By applying Lemma \ref{lemma:budget_RC} to constraints (\ref{eq:robust_model}b) and linearizing the problem, we can derive the robust counterpart for the problem $(RO')$:
\begin{subequations}
\small
\label{eq:robust_counterpart}
\begin{align}
    (RC) \quad \min_{\boldsymbol{x} \in \mathcal{X}_B, \boldsymbol{\lambda}, \boldsymbol{\nu}} \quad & \alpha \\
    \text{s.t.} \quad
    & \sum_{(o, d, t) \in \mathcal{F}} \sum_{v \in \mathcal{V}} \sum_{p \in \mathcal{P}^{o, d}} \sum_{\tau=\tau_t^{o,d,p}}^T \left( w_{t, \tau}^{o, d, p, v} + \gamma \phi^{o, d, p} \right) \lambda_{t, \tau}^{o, d, p, v} + M \sum_{(o, d, t) \in \mathcal{F}} \mu_t^{o, d} \notag \\ 
    & - M \sum_{(o, d, t) \in \mathcal{F}} \sum_{v \in \mathcal{V}} \sum_{p \in \mathcal{P}^{o, d}} \sum_{\tau=\tau_t^{o,d,p}}^T \lambda_{t, \tau}^{o, d, p, v} + \sum_{(o, d, t) \in \mathcal{F}} \nu^{o,d,t,1} + \Gamma \nu^2 \leq  \alpha; \\
    & \nu^{o,d,t,1} + \Gamma \nu^2 \geq M \sigma_t^{o,d} \quad \forall (o,d,t) \in \mathcal{F}; \\
    & \nu^{o,d,t,1} + \Gamma \nu^2 \geq - M \sigma_t^{o,d} \quad \forall (o,d,t) \in \mathcal{F}; \\
    & \nu^{o,d,t,1} \geq 0 \quad \forall (o,d,t) \in \mathcal{F}; \\
    & \nu^2 \geq 0; \\ 
    & \sum_{(o', d', t') \in \mathcal{F}} \nu_{o',d',t'}^{o,d,t,3} + \nu^{o,d,t,4} \leq \mu_t^{o, d} - \sum_{v \in \mathcal{V}} \sum_{p \in \mathcal{P}^{o, d}} \sum_{\tau=\tau_t^{o,d,p}}^T \lambda_{t, \tau}^{o, d, p, v} \quad \forall (o, d, t) \in \mathcal{F}; \\
    & \nu_{o,d,t}^{o,d,t,3} + \nu^{o,d,t,4} \geq \sigma_t^{o,d} \quad \forall (o,d,t) \in \mathcal{F}; \\
    & \nu_{o',d',t'}^{o,d,t,3} + \nu^{o,d,t,4} \geq 0 \quad \forall (o',d',t') \neq (o,d,t) \in \mathcal{F}; \\
    & \nu_{o',d',t'}^{o,d,t,3} \geq 0 \quad \forall (o',d',t'), (o,d,t) \in \mathcal{F}; \\
    & \nu^{o,d,t,4} \geq 0 \quad \forall (o,d,t) \in \mathcal{F}; \\
    & \sum_{o \in \mathcal{S}_p^{\text{before}}(s)} \sum_{d \in \mathcal{S}_p^{\text{after}}(s)} \sum_{t=1}^{T_{\tau,p}^{o,d}} \lambda_{t, \tau}^{o, d, p, v} \leq \Bar{C}_v x_\tau^{p, v} \quad \forall p \in \mathcal{P}, \forall v \in \mathcal{V}, \forall s \in \mathcal{S}_p, \forall \tau = 1,...,T; \\
    & \lambda_{t, \tau}^{o, d, p, v} \geq 0 \quad \forall (o, d, t) \in \mathcal{F}, \forall p \in \mathcal{P}, \forall v \in \mathcal{V}, \forall \tau = 1,...,T.
\end{align}
\end{subequations}

Constraints (\ref{eq:robust_counterpart}b) - (\ref{eq:robust_counterpart}f) are the robust counterpart corresponds to constraints (\ref{eq:robust_model}b) while constraints (\ref{eq:robust_counterpart}g) - (\ref{eq:robust_counterpart}k) are the robust counterpart corresponds to constraints (\ref{eq:robust_model}d).
Compared to problem $(RO')$, the robust counterpart $(RC)$ introduces $(|\mathcal{F}|^2 + 2|\mathcal{F}| + 1)$ additional auxiliary non-negative continuous variables and $(|\mathcal{F}|^2 + 2|\mathcal{F}|)$ additional inequality constraints.
When the number of distinct passenger flows $|\mathcal{F}|$ is not large (e.g., blow 1,000), the robust counterpart $(RC)$ can be directly solved by off-the-shelf ILP solvers. However, the problem $(RC)$ can be intractable when $|\mathcal{F}|$ is large (e.g., above 10,000). In the next section, we will discuss the scalability issues for 
the TFSP under a single-line context and propose methods to handle large-scale TFSPs.

\subsection{Optimization with Large-Scale Demand Matrix}

In this section, we propose the Transit Downsizing (TD) approach to reduce the problem dimensionality and increase the tractability for proposed TFSP models given a large-scale demand matrix. As complexity issues are inherent in real-world transit problems, the proposed TD approach can be generalized to other design and operation problems in transit systems.

Take a bus line and a rail line operated by CTA for instance. The inbound direction of the CTA Blue line includes 33 stations in total, which leads to $528$ distinct OD pairs for passengers. When solving the transit frequency setting problem under a one-hour time interval with 12 decision time periods of length $\Delta = 5$ min, the number of passenger flows is $|\mathcal{F}| = 6,336$. Formulating the robust counterpart $(RC)$ introduces $40,157,569$ new continuous variables, which is a large-scale problem but might still be able to solve.

On the other hand, the northbound direction of the CTA route 49 bus contains $82$ stops overall, which gives $1,176$ distinct OD pairs for passengers. Under the same setting as the Blue line, there will be $14,112$ unique passenger flows and the robust counterpart 
$(RC)$ introduces $199,176,768$ new continuous variables. The problem becomes intractable due to the excessive problem size. These two instances imply that large-scale demand matrices commonly exist in practice. Methods need to be designed to reduce the size of demand matrices in robust transit frequency setting problems.  

The TD approach consists of two components: i) an optimality-preserved dimensionality reduction component, and ii) a heuristic-based dimensionality reduction component.
The optimality-preserved component is proposed to reduce demand matrices based on the following observation: transit demand matrices are \emph{sparse} and only a subset of passenger flows are chosen by passengers. Passengers using transit services have clear spatial and temporal patterns, which lead to sparsity in demand matrices. 

\begin{prop}
\label{prop:sparsity}
    For the nominal TFSP model $(P)$ with a demand matrix $\boldsymbol{u}$, it is equivalent to solving the problem with a reduced set of passenger flow $\Bar{\mathcal{F}}$, where $\Bar{\mathcal{F}}$ only contains passenger flows with positive demand, i.e., $\Bar{\mathcal{F}} = \{ (o,d,t): u_t^{o,d} > 0 \}$.
\end{prop}

\begin{proof}
For a passenger flow $(o,d,t)$, when the demand is zero, i.e., $u_t^{o,d} = 0$, constraints (\ref{eq:flow_conservation}) ensure that $\lambda_{t,\tau}^{o,d,p,v} = 0, \forall \tau=1,...,T, p \in \mathcal{P}, v \in \mathcal{V}$, in the optimal solution given a minimization problem. Therefore, we can reach the same optimal solution by only considering passenger flows $\Bar{\mathcal{F}}$ with positive demand only, i.e., $\Bar{\mathcal{F}} = \{ (o,d,t): u_t^{o,d} > 0 \}$.
\end{proof}

Proposition \ref{prop:sparsity} reduces the problem size of the nominal TFSP model $(P)$ and $(P-C)$. It can also be applied to stochastic formulation $(SP)$ and robust formulation $(RO)$. For the stochastic TFSP model $(SP)$, each demand scenario $e \in \mathcal{E}$ with demand matrix $\boldsymbol{u}_e$ leads to a reduced passenger flow set $\Bar{\mathcal{F}_e}$, i.e., $\Bar{\mathcal{F}}_e = \{ (o,d,t): u_{t,e}^{o,d} > 0 \}$. For the robust TFSP model $(RO)$, the reduced passenger flow set $\Bar{\mathcal{F}}$ is constructed based on mean demand $\boldsymbol{\mu}$, i.e., $\Bar{\mathcal{F}} = \{ (o,d,t): \mu_t^{o,d} > 0 \}$.

The optimality-preserved component of the TD approach is extremely effective when solving nominal and stochastic models, where reduced passenger flow sets are established based on daily demand. When applying it to the robust problem with the average demand $\boldsymbol{\mu}$, the approach becomes less effective because the number of non-zero mean demand is still large. Considering the demand data from one month, a passenger flow $(o,d,t)$ has to be incorporated in $\Bar{\mathcal{F}}$ if it has demand for at least one day. We utilize a probabilistic scenario to better explain this issue. If a passenger flow $(o,d,t)$ has a 90\% probability to have zero demand in one day, the probability of not having a positive mean demand for 30 days is $0.9^{30} = 4.24\%$. When considering a month of demand data, the probability of excluding the passenger flow $(o,d,t)$ from the problem shrinks from $90\%$ to $4.24\%$, indicating that the first component of the TD approach is not effective for robust problems when considering demand data across multiple days. 

Therefore, a heuristic-based dimensionality reduction component of the TD approach is further proposed to reduce the problem size of robust TFSP model $(RO)$. It is constructed based on the following observation: if a passenger flow $(o, d, t)$ only appears once in a long period of time (e.g., one month), it is reasonable to exclude it from setting transit schedules given the same passenger flow $(o, d, t)$ will most likely not be seen again in the future. The heuristic-based component introduces an adjusted passenger flow set $\Tilde{\mathcal{F}} = \{(o,d,t): \mu_t^{o,d} > \epsilon \}$, where passenger flows with mean demand below or equal to $\epsilon$ will be excluded from the optimization model. 

The new problem after TD has a smaller scale and can be solved efficiently in practice. Compared to the original problem, the new problem has less number of constraints (i.e., a larger feasible space). Hence, its optimal objective function will be better (i.e., smaller in the minimization context). In the following analysis, we show that the difference between the objective functions of the new and original problems is bounded. The bound is a function of $\epsilon$. A smaller value of $\epsilon$ implies a tighter bound.

Define $Z^*(\mathcal{F})$ as the optimal objective function of the robust TFSP model $(RO)$ (Eq. \ref{eq:robust_model}) with passenger flow set $\mathcal{F}$. Then the optimal objective function of the problem after TD can be represented as $Z^*(\Tilde{\mathcal{F}})$. We have the following lemma:

\begin{lemma} \label{lemma:unit_loss_bound}
    For any given passenger flow set $\mathcal{F}_1$, define $\mathcal{F}_2$ as the passenger flow set by eliminating one passenger flow tuple $(o,d,t)$ (i.e., $|\mathcal{F}_1|$ - $|\mathcal{F}_2|$ = 1). Then, we have:
    \begin{align}
        Z^*(\mathcal{F}_1) - Z^*(\mathcal{F}_2) \leq 2M \cdot \ell
    \end{align}
    where $\ell = \max_{(o,d,t) \in \mathcal{F}}(\mu_t^{o,d} + \sigma_t^{o,d})$.
\end{lemma}
\begin{proof}
    When changing the passenger flow set $\mathcal{F}_1$ to $\mathcal{F}_2$ by excluding one passenger flow tuple $(o, d, t)$, the objective value of the problem $(RO)$ decreases. The reduction of the objective value is induced by two reasons: i) less demand considered in the objective function, hence less total journey time and unsatisfied penalty, and ii) reallocation of passengers given more available vehicle capacity.

    The robust TFSP model $(RO)$ minimizes the worst-case demand scenario. Therefore, we consider the worst-case objective loss when excluding one passenger flow $(o, d, t)$. For the objective loss induced by demand reduction, it is upper-bounded by $M \cdot (\mu_t^{o,d} + \sigma_t^{o,d})$, since $M$ dominants passengers' journey time and $(\mu_t^{o,d} + \sigma_t^{o,d})$ represents the largest demand for passenger flow $(o, d, t)$ defined in the uncertainty set $\mathcal{U}(\Gamma)$. Let $\ell = \max_{(o,d,t) \in \mathcal{F}}(\mu_t^{o,d} + \sigma_t^{o,d})$ and $\ell$ is a finite value since demand values in TFSP are finite integers. Then the objective loss from demand reduction is upper-bounded by $M \cdot \ell$.

    For the objective loss induced by demand reallocation, excluding one passenger flow $(o, d, t)$ equals having $(\mu_t^{o,d} + \sigma_t^{o,d})$ more vehicle capacity. The worst-case scenario is other unsatisfied passenger flows become satisfied when having more available capacity, which is upper-bounded by $M \cdot (\mu_t^{o,d} + \sigma_t^{o,d})$. Similar to the previous argument, it is upper-bounded by a finite value $M \cdot \ell$.

    Combining two sources of the objective decrease, the maximum reduction of the objective value in $(RO)$ is upper-bounded by $2M \cdot \ell$ when excluding one passenger flow $(o,d,t)$ from $\mathcal{F}_1$.
\end{proof} 

\begin{definition}
    \textbf{Dimensionality Reduction Function:} given the value of $\epsilon$ in the heuristic-based component of the TD approach, the dimensionality reduction function is defined as
    \begin{equation*}
        f(\epsilon) = \left| \{(o,d,t): \mu_t^{o,d} \leq \epsilon \} \right|,
    \end{equation*}
    which is the size of passenger flows excluded from $\mathcal{F}$. The dimensionality reduction function $f(\epsilon)$ has the following properties:
    \begin{enumerate}
        \item $f(\epsilon = 0) = 0$ (assuming all $\mu_t^{o,d} > 0$) and $\lim_{\epsilon \rightarrow \infty}f(\epsilon) = |\mathcal{F}|$.
        \item $f(\epsilon)$ monotonically increases when $\epsilon$ increases.
        \item $0 \leq f(\epsilon) \leq |\mathcal{F}| < +\infty$. 
    \end{enumerate}
\end{definition}

The first property holds because we do not exclude any passenger flows with when $\epsilon = 0$, and all passenger flows are excluded when $\epsilon$ is a large enough value. The second property holds since more passenger flows will be excluded when increasing $\epsilon$. The last property is directly derived from the first two. Note that $f(\epsilon)$ is finite because the total number of passenger flows is finite considering a finite network and time interval in practice. By defining the dimensionality reduction function, we have the following proposition:

\begin{prop} \label{prop:finite_loss_bound}
    For the robust TFSP model $(RO)$ applying the TD approach, the objective reduction is upper-bounded by a finite value $\Lambda (\epsilon) =  2M \cdot \ell\cdot  f(\epsilon)$. Mathmatically:
    \begin{align}
        Z^*(\mathcal{F}) - Z^*(\Tilde{\mathcal{F}}) \leq 2M \cdot \ell\cdot  f(\epsilon)
    \end{align}
    $\Lambda (\epsilon)$ has the following properties:
    \begin{enumerate}
        \item $\Lambda (\epsilon = 0) = 0$.
        \item $\Lambda(\epsilon)$ monotonically increases when $\epsilon$ increases.
    \end{enumerate}
\end{prop}
\begin{proof}
    Lemma \ref{lemma:unit_loss_bound} implies that the objective reduction due to excluding one passenger flow tuple $(o,d,t)$ from $\mathcal{F}$ is upper-bounded by $2M\ell$. The size of passenger flow tuples excluding from $\mathcal{F}$ given $\epsilon$ is $f(\epsilon)$. Therefore, the objective reduction is upper-bounded by $2M \ell f(\epsilon)$, which is a finite value since $f(\epsilon)$ is upper-bounded by $|\mathcal{F}|$. Define $\Lambda(\epsilon) = 2M \ell f(\epsilon)$ and we have shown the objective loss $\Lambda(\epsilon)$ is upper-bounded.

    According to the definition of dimensionality reduction function, when $\epsilon = 0$, we have $ f(\epsilon = 0) = 0$, thus $\Lambda (\epsilon = 0) = 0$. Moreover, since $ f(\epsilon)$ monotonically increases when $\epsilon$ increases, $\Lambda(\epsilon)$ also monotonically increases when $\epsilon$ increases. 
\end{proof}

Proposition \ref{prop:finite_loss_bound} indicates that the objective change due to the heuristic-based component of the TD approach is upper-bounded by a finite value. Meanwhile, decreasing the value of $\epsilon$ leads to a tighter bound. This shows that our proposed TD method is a valid approximation of the original problem with bounded errors. This proposition is validated with the experiments on the sensitivity analysis of $\epsilon$ in Section \ref{subsec:uncertain_model}.   

Setting the value of $\epsilon$ is critical in the proposed method. The value of $\epsilon$ should be chosen to balance the trade-off between transit schedule performance and problem complexity. Let $m$ represent the number of days considered in the problem. The proposed heuristic approach works well in practice when setting $\epsilon = \frac{1}{m}$, indicating that passenger flows that appear only once over $m$ days will be excluded from the problem.  

Overall, the proposed TD approach helps to solve TFSPs with large-scale demand matrices. The first component maintains optimality and the second heuristic-based component could lead to sub-optimal solutions.

\section{Results}
\label{sec:results}

In this section, the numerical results of the proposed models will be covered. All experimental results in this paper were generated on a machine with a 3.0 GHz AMD Threadripper 2970WX Processor and 128 GB Memory. The linear programs in the experiments for generating optimal transit schedules and evaluating solution performances were solved with Gurobi 9.0.3~\cite{gurobi}.

The results section is organized as follows. 
Section \ref{subsec:data_description} describes data, parameter values, and experimental setups.
Section \ref{subsec:base_model} displays performance comparisons between the optimized schedule w/o considering demand uncertainty and the current schedule. Sensitivity analyses and crowding extensions are also discussed in this section. 
Section \ref{subsec:uncertain_model} shows performance comparisons between current, stochastic, and robust transit schedules.  

\subsection{Data Description}
\label{subsec:data_description}

\begin{table}[h!]
    \small
    \centering
    \resizebox{\textwidth}{!}{
    \begin{tabular}{ l | l | c }
     \hline 
     Model Parameter & Explanation & Base Case Value\\
     \hline
     $T_{start}$ & Start time of planning period & 07:00  \\
     $T_{end}$ & End time of planning period & 09:00  \\
     $\Delta$ & Decision time interval length & 5 (minutes)  \\
     $T$ & Number of decision time periods & 24 \\
     $\mathcal{P}$ & Set of patterns for the transit line & \{49, X49\} \\
     $\mathcal{V}$ & Set of bus types & \{standard, articulated\}\\
     $C_v$ & Number of seats on buses & \{37, 58\} \\
     $\Bar{C}_v$ & Maximum vehicle capacity & \{70, 107\}\\
     $c^{p,v}, \forall p \in \mathcal{P}, \forall v \in \mathcal{V}$ & Cost parameter for bus with pattern $p$ and vehicle type $v$ & 1 \\
     $B$ & Total vehicle budget during the planning period & 20 \\
     $M$ & Penalty for an unsatisfied passenger & $10^5$ \\
     $\gamma$ & Weight parameter for in-vehicle travel time & 1\\
     $m$ & Number of demand scenarios & 22 \\
     $\epsilon$ & Heuristic parameter for demand matrix size reduction & 0.05 \\
     \hline
     \end{tabular}}
    \caption{Model parameters and base case value.}
    \label{tab:model_parameter}
\end{table}

Parameter values used in the experiments are shown in Table \ref{tab:model_parameter}.
The study transit lines used in the experiments are Route 49 northbound and Route X49 northbound operated by the CTA. Route 49 and Route X49 both serve Western Avenue in western Chicago. Route X49 is an expressed version of Route 49 with limited stops. Route 49 has 82 bus stops and Route X49 has 35 bus stops. Both routes share the same terminals and connect multiple rail line services: Orange, Pink, Green, Blue, and Brown lines. 

In practice, transit schedules for Route 49 and Route X49 are determined separately. In our proposed optimization model, we will consider two routes as two patterns for a single transit line and generate both schedules simultaneously, i.e., $\mathcal{P} = \{ 49, X49 \}$. The position of both routes within the CTA transit network and stop overviews are shown in Figure \ref{fig:route_49_X49}.

\begin{figure}[!h]
    \centering
    \includegraphics[width=\textwidth]{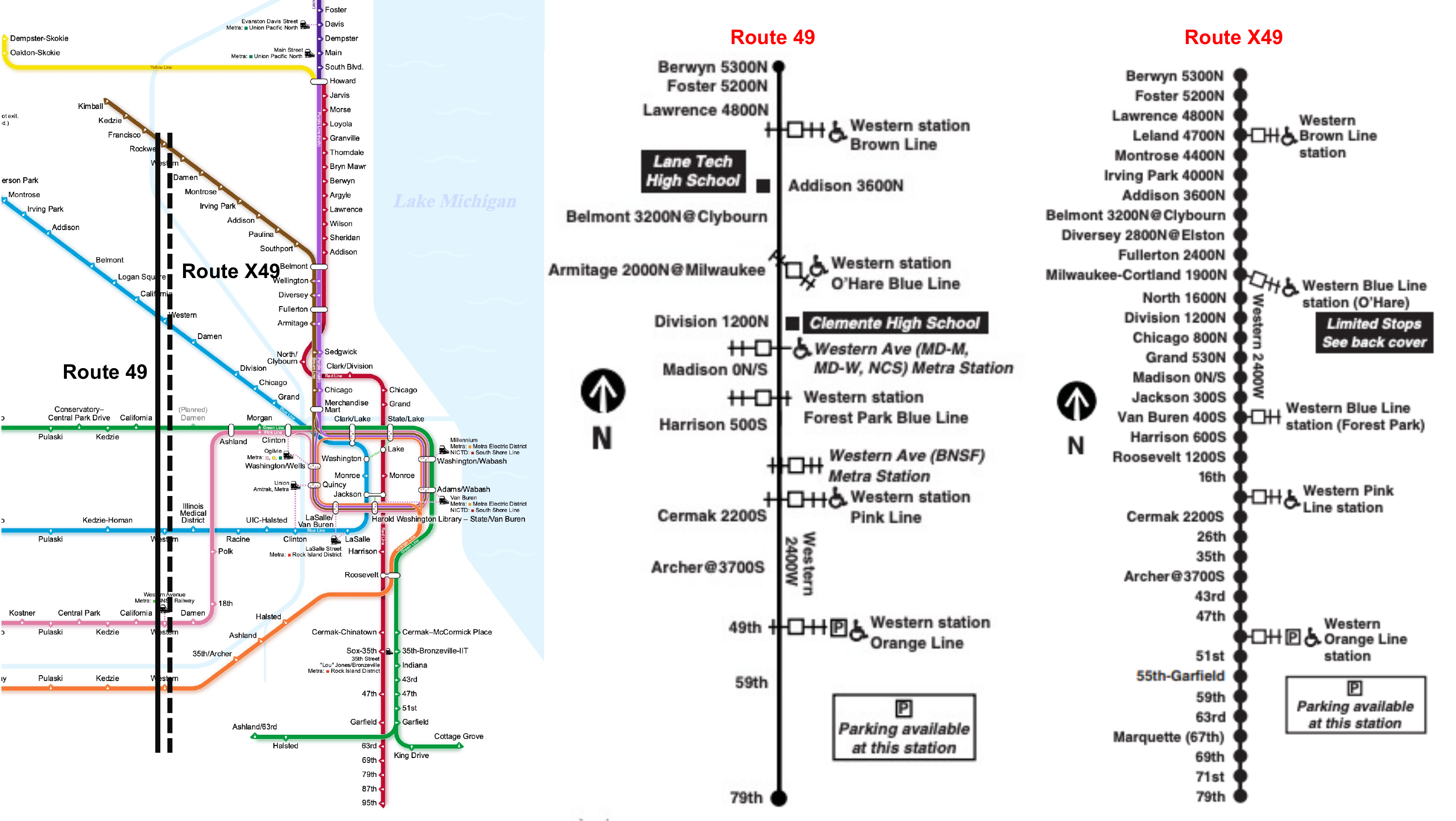}
    \caption{Positions and stop overviews of Route 49 and Route X49 in the CTA network.}
    \label{fig:route_49_X49}
\end{figure}

The data utilized in the experiments are from 22 weekdays in October 2020. The current transit schedule information is from an open-source Generalized Transit Feed Specification (GTFS) dataset, which is published by CTA every month. Regarding the running times between any two stops for different patterns, they are calculated based on the Automatic Vehicle Location (AVL) dataset of October 2020 provided by CTA. The OD matrix is generated based on CTA's ODX dataset from October 2020. 

The ``ODX'' stands for ``origin, destination, and transfer inference algorithm'', an algorithm developed by~\citet{ODX_Gabriel} and currently implemented within the CTA. The CTA transit network is equipped with a ``tap-on'' only fare collection system, indicating that alighting information is not reported in the system. The ODX algorithm is utilized to infer the alighting information and details can be found in \cite{ODX_Gabriel, ROVE_2022, ODX_Zhao}. 

The study period is a two-hour time interval from 7:00 AM to 9:00 AM. The length of each decision time interval is $\Delta = 5$ minutes, therefore, there are 24 time intervals considered in the transit frequency setting problem. For the existing transit schedule, there are 20 buses operating in total. The current northbound schedules for the study transit line are shown in Figure \ref{fig:current_schedules}.

\begin{figure}[!h]
    \centering
    \includegraphics[width=\textwidth]{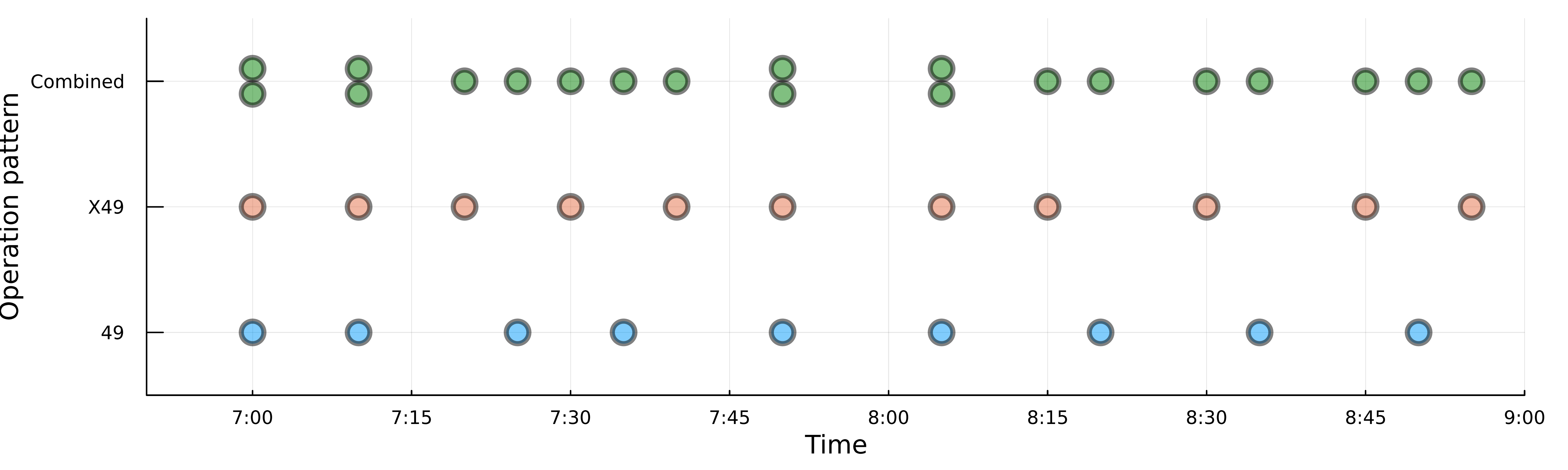}
    \caption{The current northbound transit schedule for Route 49, Route X49, and the combined transit line. Each colored dot represents a departure with a specific operation pattern from the terminal stop. }
    \label{fig:current_schedules}
\end{figure}

In the experiments, the budget constraint in (\ref{eq:feasible_x}) ensures that the total number of buses operating within the overall time interval does not exceed the maximum bus supply, i.e., $c^{p,v} = 1, \forall p \in \mathcal{P}, \forall v \in \mathcal{V}$, and $B = 20$.

For buses used in the experiments, we consider two types of buses: regular buses and articulated buses, i.e., $\mathcal{V} = \{regular, articulated\}$. The regular bus has 37 seats and a maximum capacity is 70, while the articulated bus has 58 seats with a maximum capacity of 107. The current schedule only utilizes regular buses for Route 49 and Route X49. Therefore, only regular buses are considered in the base case scenario.

\subsection{Baseline Model Performances}
\label{subsec:base_model}

\subsubsection{Optimal Transit Schedules}

To evaluate the performances of the nominal TFSP model $(P)$, we randomly choose a demand scenario from 22 weekdays to generate the optimal transit schedule, which is then compared with the current schedule over the remaining 21 demand scenarios. For the base case scenario, wait and travel times are equally important, i.e., $\gamma = 1$. The TD approach w/o the heuristic-based component is applied when solving the optimization model.

\begin{figure}[!h]
    \centering
    \includegraphics[width=\textwidth]{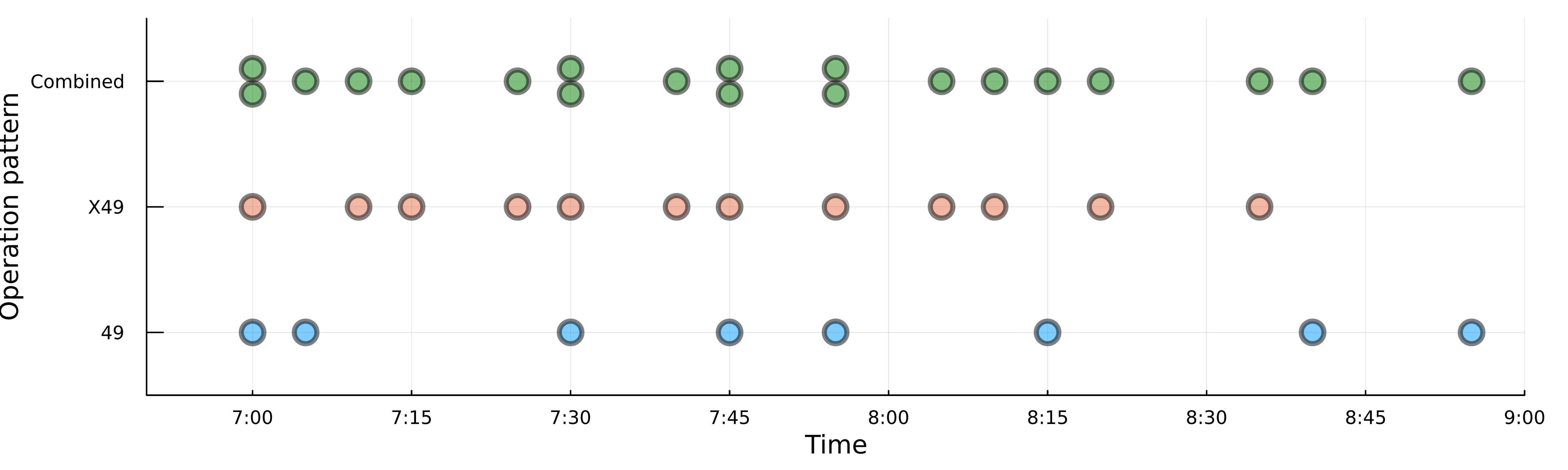}
    \caption{The optimized transit schedule w/o considering demand uncertainty based on a one-day demand scenario. Each colored dot represents a departure with a specific operation pattern from the terminal stop. }
    \label{fig:nominal_optimal_schedules}
\end{figure}

Figure \ref{fig:nominal_optimal_schedules} shows the optimized transit schedule w/o considering demand uncertainty based on a randomly-selected one-day demand scenario. Compared to the current schedule shown in Figure \ref{fig:current_schedules}, more buses are dispatched during the first hour. The optimized transit schedule w/o considering demand uncertainty becomes irregular due to serving a specific demand scenario. Meanwhile, it shifts one bus from Route 49 to Route X49.

\begin{figure}[!h]
    \centering
    \includegraphics[width=0.8\textwidth]{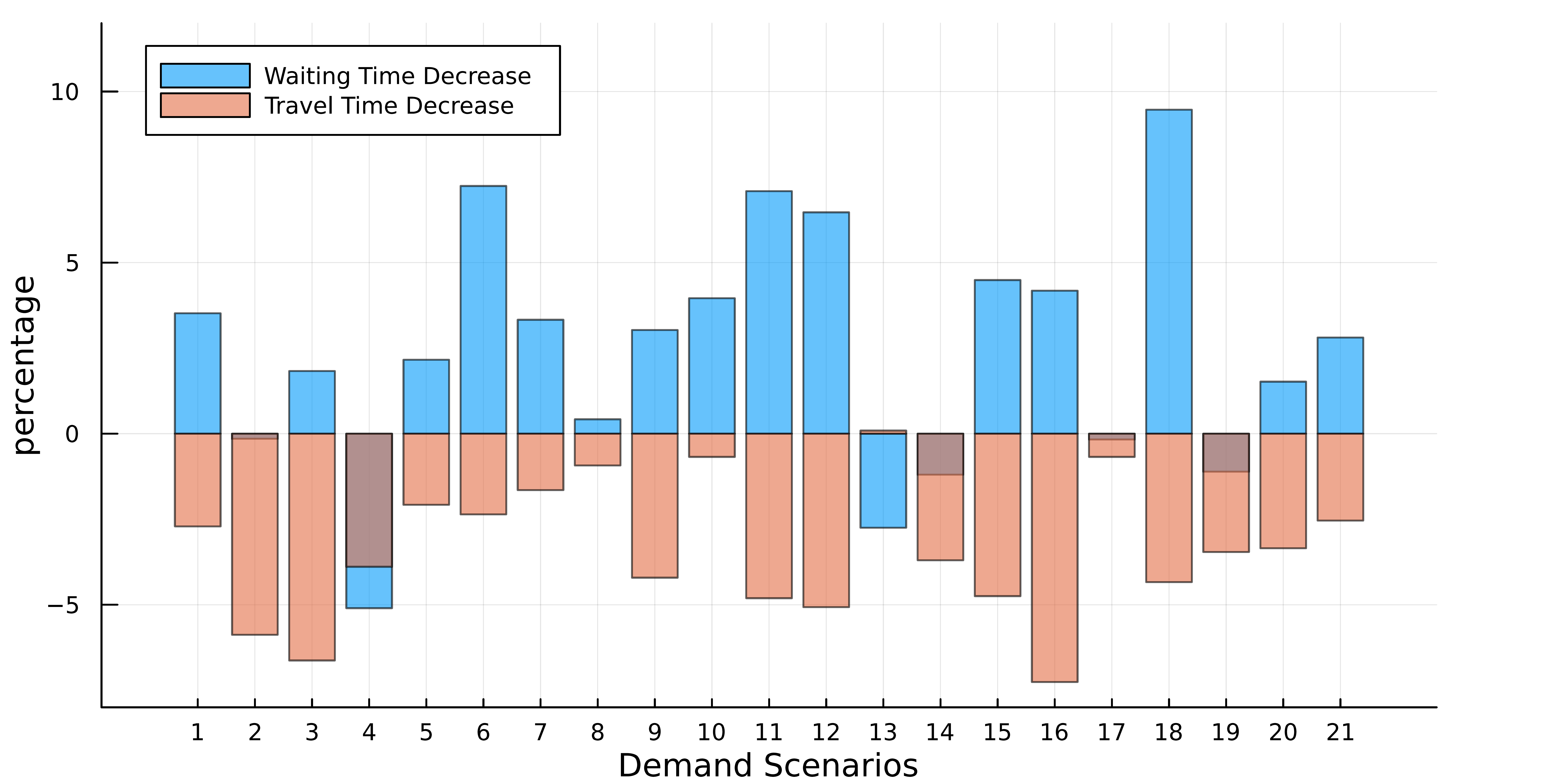}
    \caption{Performance comparisons between the current and the optimized transit schedules w/o considering demand uncertainty. Blue bars represent wait time decrease for the optimized transit schedule w/o considering demand uncertainty. Orange bars indicate travel time decrease for the optimized transit schedule w/o considering demand uncertainty.}
    \label{fig:base_perform_comparison}
\end{figure}

The performance comparison over 21 demand scenarios is shown in Figure \ref{fig:base_perform_comparison}. Bars indicate wait and travel time decreases for the optimal schedule compared to the current schedule. For the optimized transit schedule w/o considering demand uncertainty, passengers experience lower wait times in 15 out of 21 demand scenarios. However, passengers have higher in-vehicle travel times for almost all demand scenarios given the current transit schedule. In summary, a 2.43\% wait time decrease and a 3.38\% travel time increase are brought to passengers on average when switching from the current schedule to the optimized schedule w/o considering demand uncertainty. It works best for the input demand scenario of the optimization model. For other demand scenarios, it reduces passengers' wait times by sacrificing in-vehicle travel times. 

The performance comparison indicates that demand uncertainty is crucial when generating transit schedules. The optimized transit schedule w/o considering demand uncertainty does not have an edge over the existing transit schedule, which maintains a regular headway. 

\subsubsection{Crowding Extensions} \label{subsubsec:crowding}

Next, we will discuss the crowding extension of the nominal TFSP model $(P - C)$. Existing demand scenarios from October 2020 lead to very few crowded transit vehicles. Therefore, model performances will be tested based on a synthetic demand scenario with an expanded demand level.
The synthetic demand data is generated as follows: for each passenger flow $(o, d, t)$ with a non-zero average demand value $\mu_t^{o,d}$ over 22 demand scenarios, generate the new demand level according to a Poisson distribution $u_t^{o,d} \sim Pois(\beta \cdot \mu_t^{o,d})$, where $\beta$ indicates an expansion factor. In the following discussion, we generate a synthetic demand scenario with an expansion factor $\beta = 4$.

\begin{figure}[!h]
    \centering
    \includegraphics[width=0.8\textwidth]{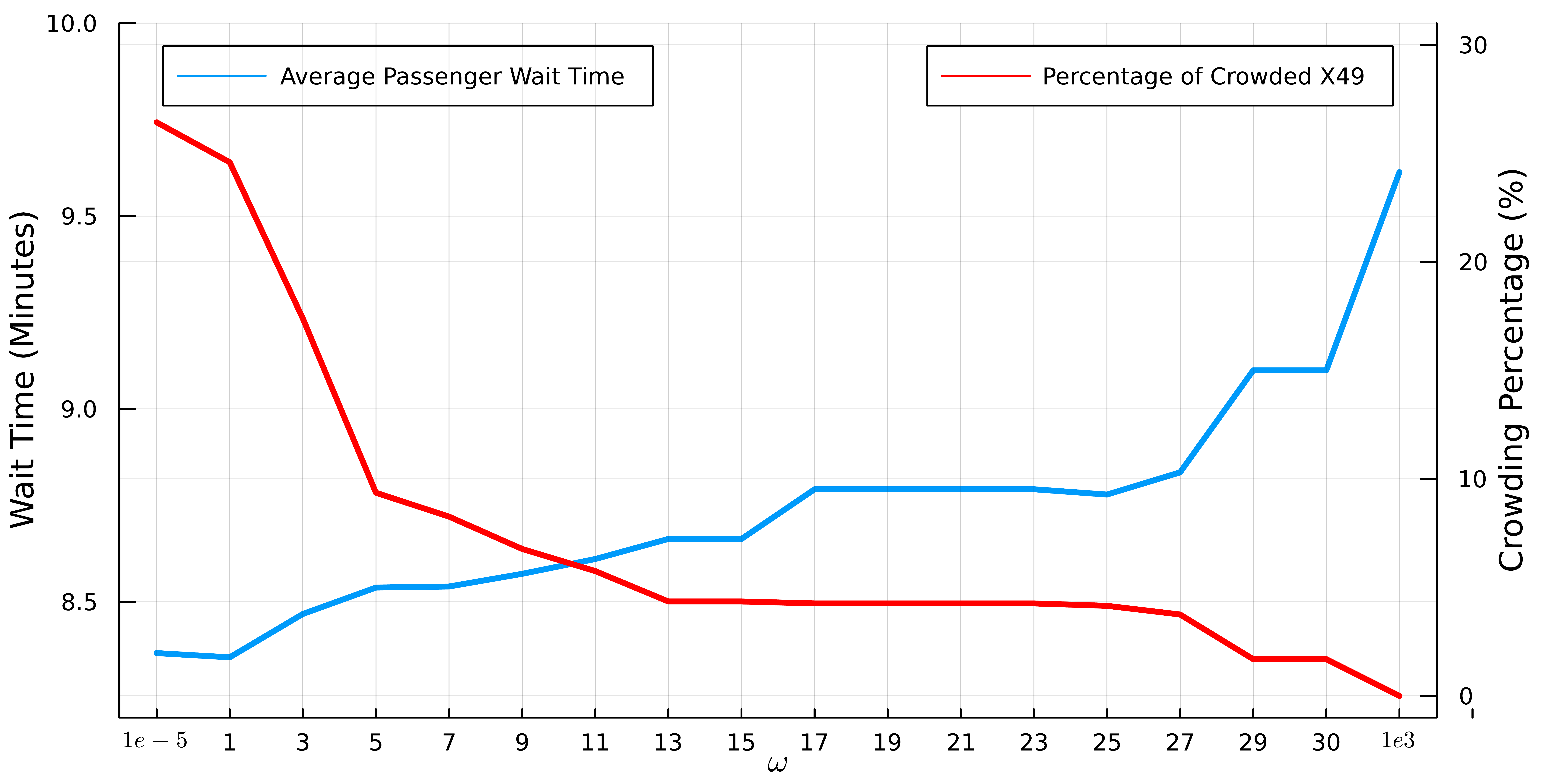}
    \caption{Trade-offs between average passenger wait times and crowding levels given different $\omega$ values.}
    \label{fig:crowding_tradeoff}
\end{figure}

In the crowding-extended model $(P - C)$, parameter $\omega$ is utilized to control the level of penalty for crowded transit vehicles in the objective function. Figure \ref{fig:crowding_tradeoff} shows the average passenger wait time and percentage of crowded X49 given different values of $\omega$. 
For the base case scenario $(\omega = 10^{-5})$ with the expanded demand scenario, 26.43\% of running time for transit vehicles operated on pattern X49 is crowded while 2.95\% of pattern 49 running time is crowded. The average passenger wait time is 8.37 minutes. 
When increasing the crowding penalty $\omega$, the crowding level on pattern X49 decreases while the average passenger wait time increases. When the value of $\omega$ exceeds a certain threshold, all passengers can have seats on buses and the average passenger wait time increases to 9.61 minutes, which is increased by 14.81\%.

It is worth noting that the crowding level is reduced by the purposely left-behind behaviors of passengers. However, passengers will always board the first available transit vehicle in reality. One way to resolve this conflict is by introducing articulated buses with a larger seat capacity. Figure \ref{fig:articulated_buses} displays the crowding percentage of pattern X49 given different numbers of available articulated buses. Introducing 6 additional articulated buses reduces the percentage of running time on pattern X49 with crowded transit vehicles to 8.72\%. The optimized transit schedule with articulated buses is shown in Figure \ref{fig:articulated_buses_schedule}. To better reduce the crowding on buses, articulated are dispatched within the first hour when more passengers are taking transit services.

\begin{figure}[!h]
    \centering
    \includegraphics[width=0.8\textwidth]{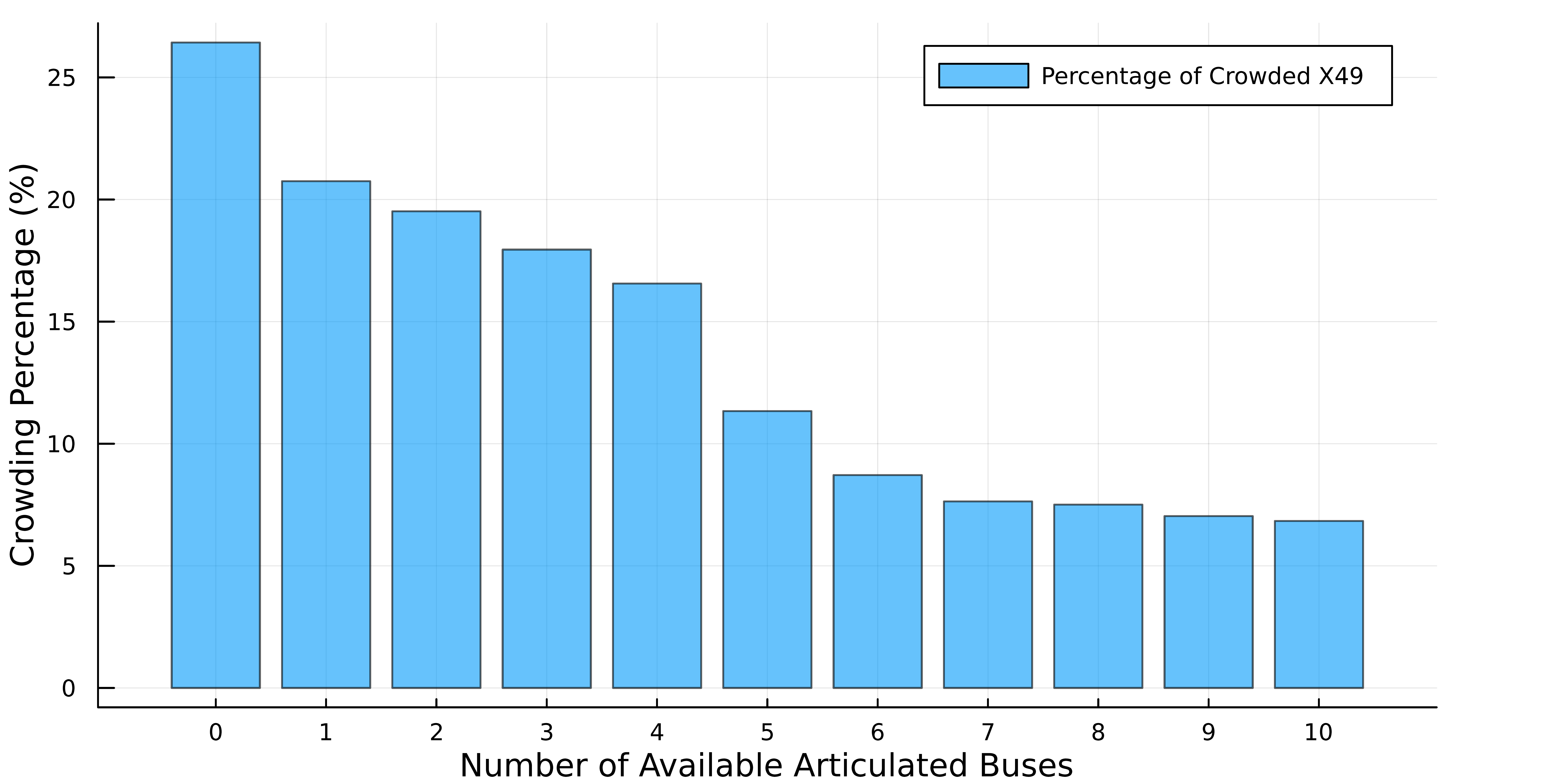}
    \caption{Trade-offs between average passenger wait times and crowding levels given different $\omega$ values.}
    \label{fig:articulated_buses}
\end{figure}

\begin{figure}[!h]
    \centering
    \includegraphics[width=\textwidth]{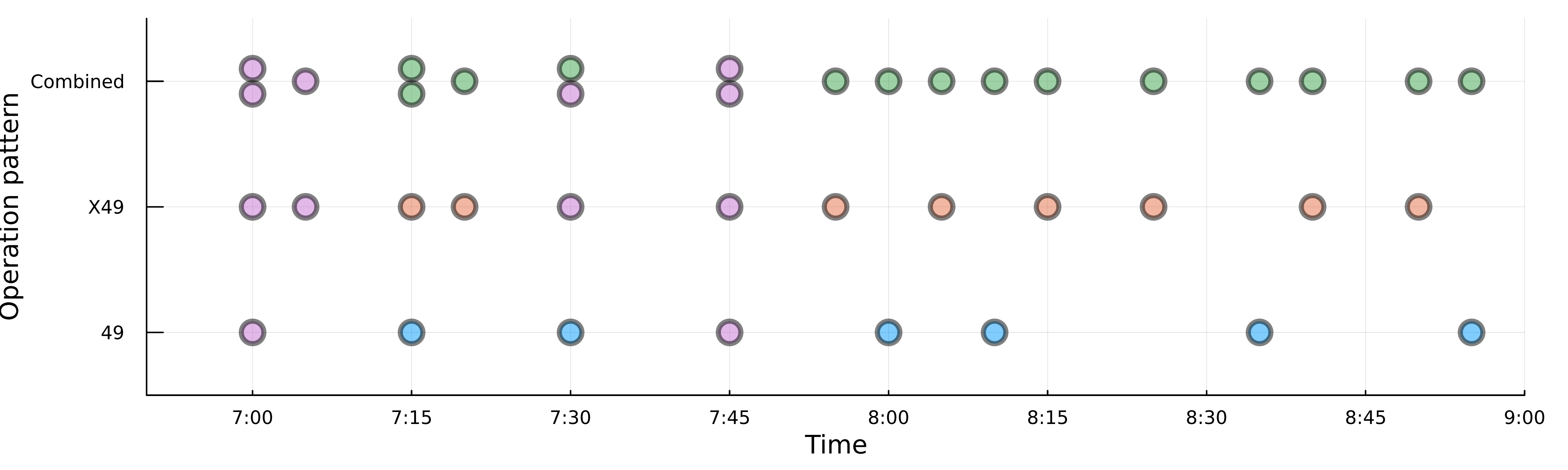}
    \caption{The optimal transit schedule with an expanded demand matrix and 6 available articulated buses. Each colored dot represents a departure with a specific operation pattern from the terminal stop. Each pink dot indicates a departure of an articulated bus from the terminal stop.}
    \label{fig:articulated_buses_schedule}
\end{figure}

The marginal benefit of bringing extra articulated buses drops significantly after having 6 articulated buses. For the scenario with 10 available articulated buses, the crowding percentage on pattern X49 is 7.04\%. In summary, having a small fleet of articulated buses can reduce the crowding levels on buses significantly in bus operations.  

\subsubsection{Sensitivity Analyses}

Lastly,  we test the sensitivity of the results when changing the weight parameter $\gamma$ for in-vehicle travel times. In previous experiments, $\gamma = 1$ was used as a base case, leading to a transit schedule that minimizes the total journey time. In this section, different values of $\gamma$ ranging from 0 to 2 with a 0.1 step size are tested. Results are shown in Figure \ref{fig:sensitivity_analyses}.

\begin{figure}[!h]
\centering
\begin{subfigure}{.495\textwidth}
  \centering
  \includegraphics[width=\linewidth]{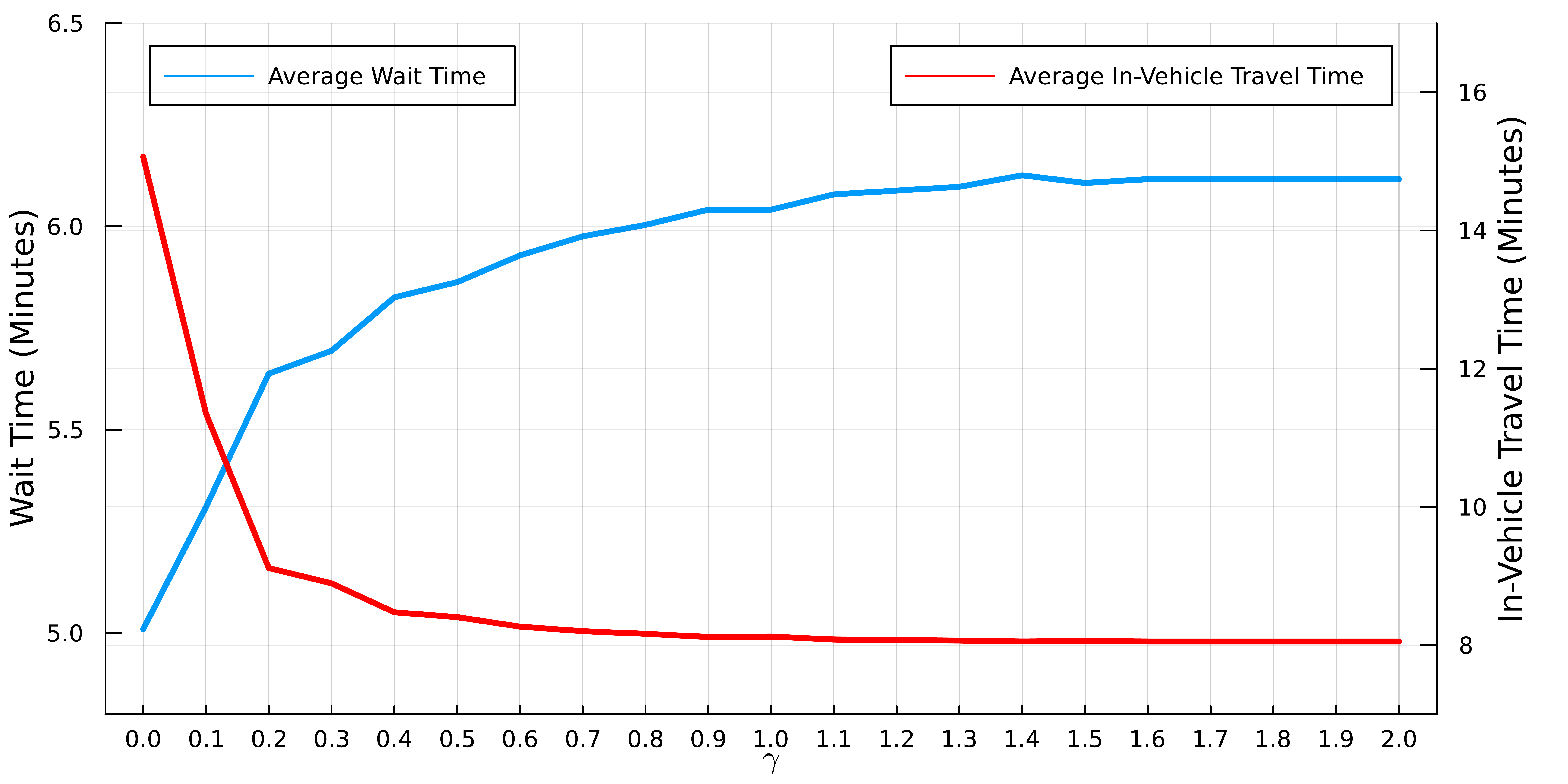}
  \caption{Average passenger wait time and average passenger in-vehicle travel time changes.}
  \label{fig:sensitivity_analyses_1}
\end{subfigure}
\begin{subfigure}{.495\textwidth}
  \centering
  \includegraphics[width=\linewidth]{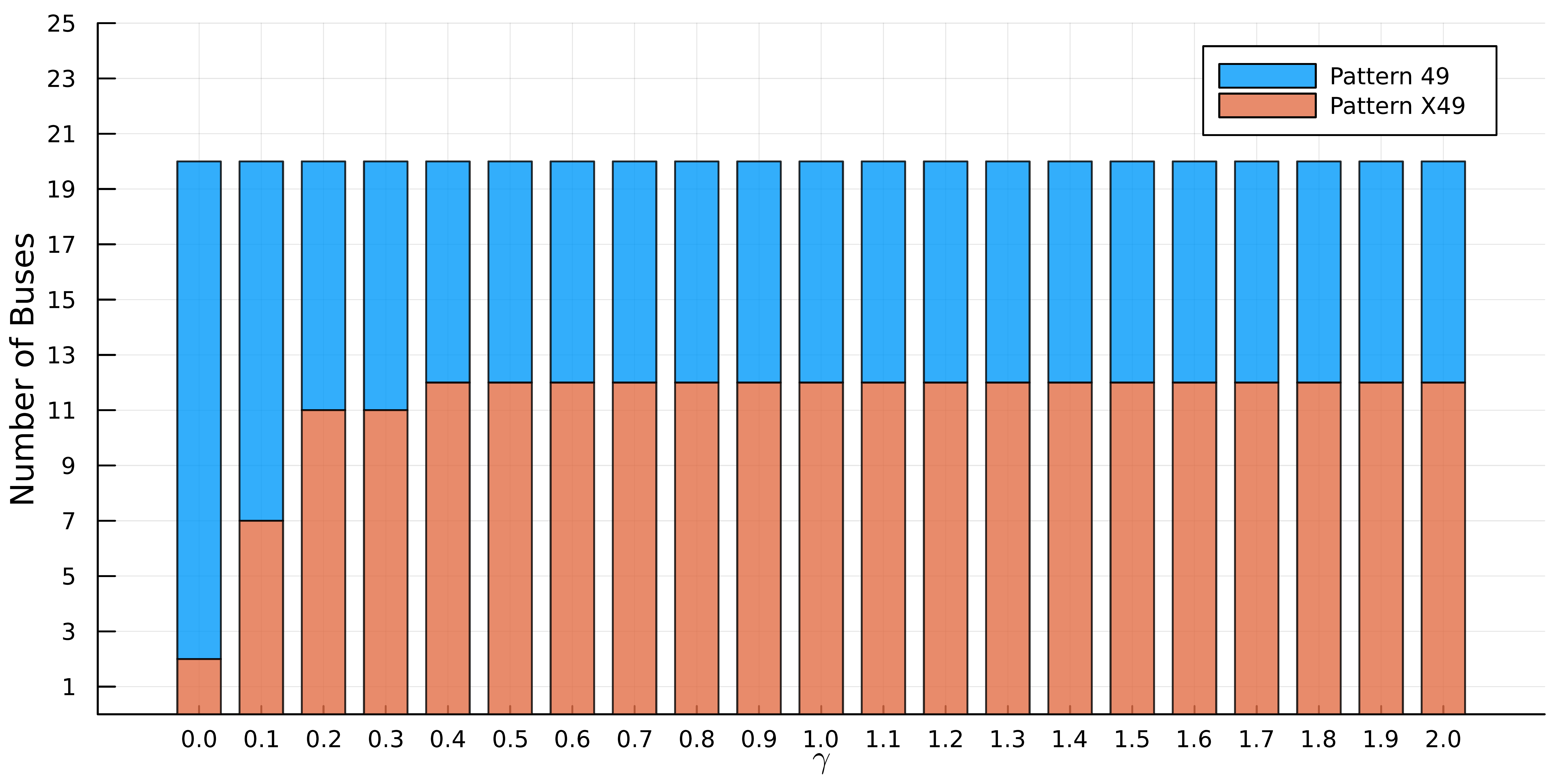}  
  \caption{Number of buses operated with pattern 49 and X49.}
  \label{fig:sensitivity_analyses_2}
\end{subfigure}
\caption{Sensitivity analyses results for the weight parameter $\gamma$.}
\label{fig:sensitivity_analyses}
\end{figure}

A smaller value of $\gamma$ indicates that wait times are more important than in-vehicle travel times. For the scenario with $\gamma = 0$, where transit schedules solely minimize passengers' wait times, the average wait time is 5.01 minutes and the average in-vehicle travel time is 15.07 minutes. The average wait time monotonically increases and the average in-vehicle travel time monotonically decreases when the value of $\gamma$ increases, which is shown in Figure \ref{fig:sensitivity_analyses_1}. For the scenario with $\gamma = 2$, where in-vehicle travel times are twice important as wait times, the average wait time is 6.12 minutes and the average in-vehicle travel time is 8.05 minutes. 

The average total travel time decreases from 20.08 minutes to 14.17 minutes when increasing $\gamma$ from 0 to 2. This is intuitive; more vehicles will be operated with pattern X49 when increasing $\gamma$, and pattern X49 has a larger vehicle speed than pattern 49 given fewer bus stops. Figure \ref{fig:sensitivity_analyses_2} shows the number of buses running on each pattern given different values of $\gamma$. Only 2 bus with pattern X49 is operated when $\gamma = 0$, while 12 buses with pattern X49 are operated when $\gamma$ becomes larger.

\subsection{Stochastic and Robust Model Performances}
\label{subsec:uncertain_model}

To incorporate demand uncertainty into the TFSP, the stochastic TFSP model $(SP)$ and the robust TFSP model $(RO)$ are proposed. In this section, we will compare the performances of each model with the current transit schedule over multiple synthetic demand scenarios. The synthetic demand scenario is generated following the method described in Section \ref{subsubsec:crowding} with no demand expansion, i.e., $\beta = 1$.

For the stochastic transit schedule, it is generated by the TD approach w/o the heuristic-based component. Figure \ref{fig:stochastic_model_performances} shows the performance comparison between stochastic and current transit schedules over 50 randomly-generated demand scenarios. On average, the stochastic schedule improves passengers' wait time by 4.71\% and in-vehicle travel time by 0.80\%. An optimized transit schedule over 22 demand scenarios is more robust than an optimized transit schedule with only one demand scenario. The stochastic transit schedule improves both wait and in-vehicle travel times in 41 out of 50 demand scenarios.

\begin{figure}[!h]
    \centering
    \includegraphics[width=0.8\textwidth]{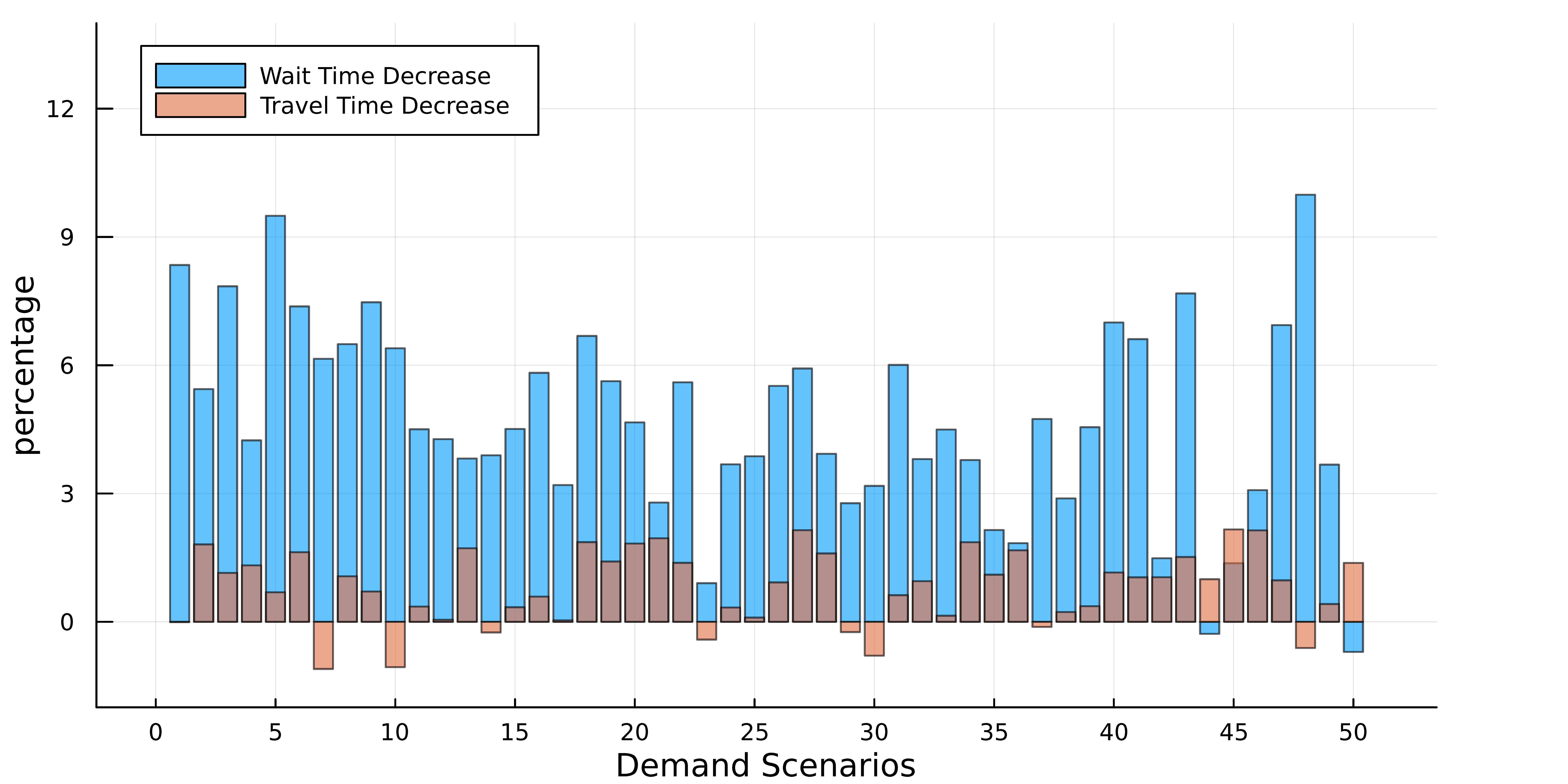}
    \caption{Performance comparisons between the current and the stochastic transit schedules over 50 randomly generated demand scenarios. Blue bars represent wait time decrease for the stochastic transit schedule. Orange bars indicate travel time decrease for the stochastic transit schedule.}
    \label{fig:stochastic_model_performances}
\end{figure}

Figure \ref{fig:stochastic_optimal_schedules} shows the stochastic transit schedule. Compared to the current transit schedule shown in Figure \ref{fig:current_schedules}, it has fewer time intervals where buses are dispatched for both patterns. In the combined transit schedule, buses are spread more evenly during the two-hour decision time period. Meanwhile, one additional bus is operated with pattern X49. Compared to the optimal transit schedule with one-day demand displayed in Figure \ref{fig:nominal_optimal_schedules}, the stochastic transit schedule maintains a stable headway for both patterns, which is similar to the current schedule, where the headway-based transit operation strategy is utilized.

\begin{figure}[!h]
    \centering
    \includegraphics[width=\textwidth]{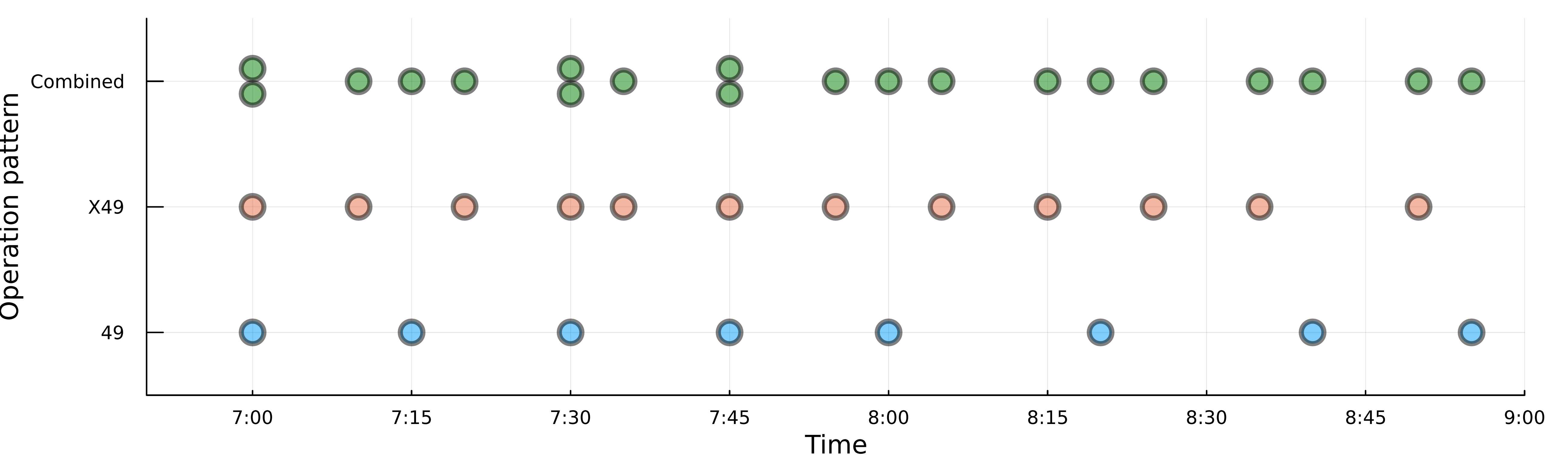}
    \caption{The stochastic transit schedule based on 22 demand scenarios. Each colored dot represents a departure with a specific operation pattern from the terminal stop. }
    \label{fig:stochastic_optimal_schedules}
\end{figure}

For the robust transit schedule, it is generated by the TD approach with $\epsilon = 0.05$, meaning that a passenger flow $(o,d,t)$ will be incorporated in the model only if it appears more than one time within 22 weekdays. The robust optimization model is solved by the off-the-shelf MIP (Mixed Integer Programming) solver Gurobi with a 3-hour time limit and an optimality gap of 0.5\%. Results are shown in Table \ref{tab:robust_results}.

\begin{table}[h!]
\centering
\begin{tabular}{ c | c c c | c c c | c  }\hline
{$\Gamma$} & Wait Time & Compare & Improve & Travel Time & Compare & Improve & GAP \\ \hline
0.0 & 7.749 & -0.98\% & 3.79\% & 8.456 & 0.13\% & 0.93\% & OPT\\
1.0 & 7.706 & -0.4\% & 4.34\% & 8.534 & -0.79\% & 0.02\% & 8.51\%\\
2.0 & 7.858 & -2.42\% & 2.42\% & 8.443 & 0.29\% & 1.09\% &3.56\%\\
3.0 & 7.942 & -3.5\% & 1.4\% & 8.425 & 0.49\% & 1.29\% &2.40\%\\
4.0 & 7.816 & -1.87\% & 2.94\% & 8.435 & 0.38\% & 1.18\% &1.38\%\\
5.0 & 7.779 & -1.39\% & 3.4\% & 8.425 & 0.5\% & 1.3\% &1.08\%\\
6.0 & 7.784 & -1.46\% & 3.34\% & 8.429 & 0.46\% & 1.26\% &0.75\%\\
7.0 & 7.784 & -1.46\% & 3.34\% & 8.429 & 0.46\% & 1.26\% &0.51\%\\
8.0 & 7.784 & -1.46\% & 3.34\% & 8.429 & 0.46\% & 1.26\% & OPT\\
9.0 & 7.736 & -0.84\% & 3.93\% & 8.443 & 0.29\% & 1.09\% & OPT\\
10.0 & 7.783 & -1.44\% & 3.36\% & 8.433 & 0.4\% & 1.2\% & OPT\\
\hline
\end{tabular}
\caption{Performance evaluations for robust transit schedules. $\Gamma$ indicates a parameter for controlling the size of budget uncertainty sets. $Wait \; Time$ and $Travel \; Time$ represent the average wait time and travel time for passengers over 50 randomly generated demand scenarios. $Compare$ indicates the performance comparison with the stochastic transit schedule. $Improv$ stands for the performance comparison with the current transit schedule. $GAP$ is the optimality gap for the MIP solver.}
\label{tab:robust_results}
\end{table}

Parameter $\Gamma$ controls the level of demand uncertainty incorporated in the model. A higher value of $\Gamma$ indicates that more demand uncertainty is considered when generating the robust transit schedule. When $\Gamma = 0$, the robust optimization is reduced to the nominal optimization model with the mean demand matrix $(\mu_t^{o,d})$ as the model input. For all uncertain scenarios, robust transit schedules outperform the current transit schedule by reducing both wait times and in-vehicle travel times. Compared to the stochastic transit schedule, robust transit schedules have better in-vehicle travel times and worse wait times for passengers.  

When increasing the value of $\Gamma$ in the model, the robust optimization model becomes easier to be solved as the optimality gap becomes smaller. The model can be solved optimally when $\Gamma$ is greater than 7. This can be explained as follows: a larger value of $\Gamma$ leads to a less-restricted optimization problem; heuristic approaches implemented in Gurobi are more likely to produce feasible solutions; better heuristic solutions reduce the time for branch-and-bound significantly. With respect to the model performance, it does not have a pattern regarding the uncertain parameter $\Gamma$. The robust transit schedule with $\Gamma = 10$ is shown in Figure \ref{fig:robust_optimal_schedules}. Other robust transit schedules are shown in \ref{appendix:robust_schedules}.

\begin{figure}[!h]
    \centering
    \includegraphics[width=\textwidth]{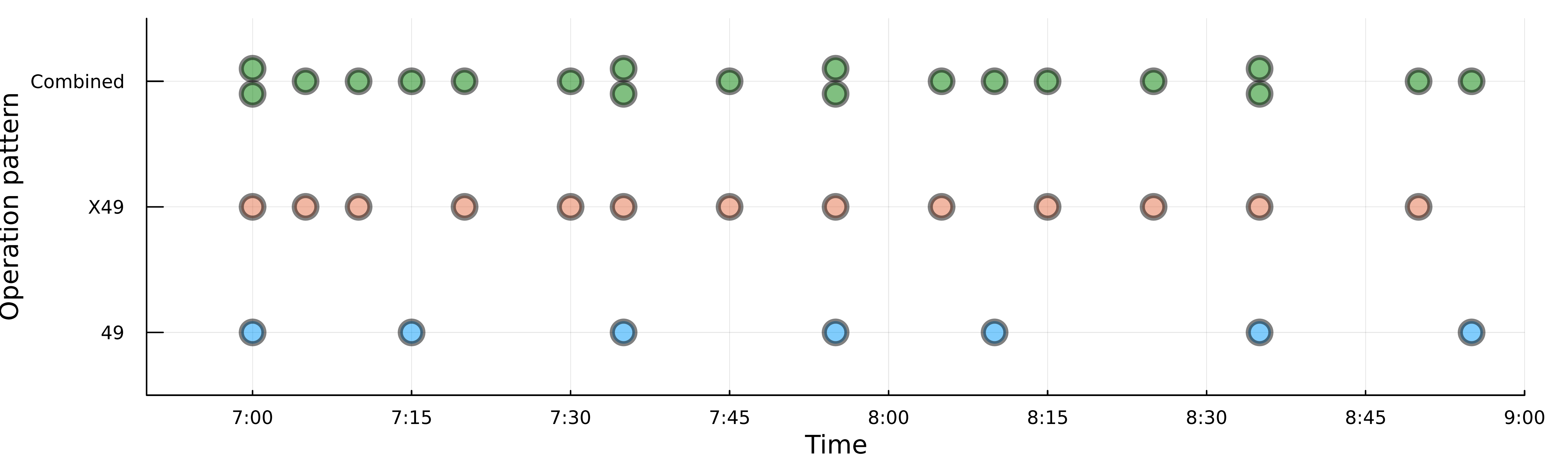}
    \caption{The robust transit schedule with $\Gamma = 10$. Each colored dot represents a departure with a specific operation pattern from the terminal stop.}
    \label{fig:robust_optimal_schedules}
\end{figure}

Compared to the stochastic transit schedule shown in Figure \ref{fig:stochastic_optimal_schedules}, the robust transit schedule utilizes one more bus over pattern X49. Meanwhile, more buses are dispatched during the first hour from the terminal. In summary, the robust transit schedule has a competitive performance over the stochastic transit schedule. Robust transit schedules can be adopted when vehicles are crowded and passengers prefer less in-vehicle travel times. The uncertain parameter $\Gamma$ in the model needs to be selected carefully to reflect the actual demand uncertainty. Advanced data-driven robust optimization approach with the ability to automatically select uncertain parameter $\Gamma$ can be further introduced~\cite{Bertsimas2018}.

\begin{figure}[!h]
    \centering
    \includegraphics[width=0.7\textwidth]{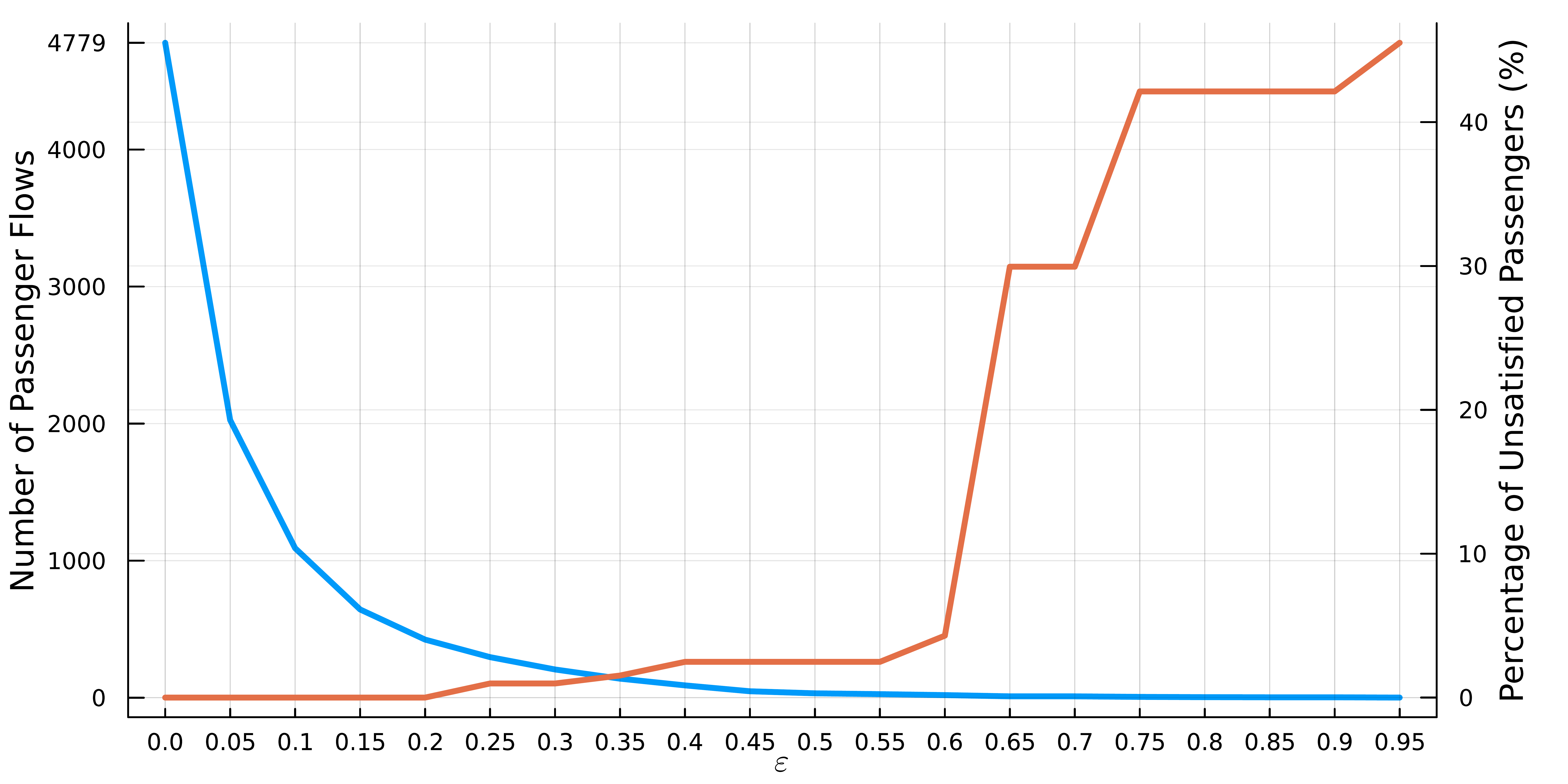}
    \caption{Sensitivity analyses for parameter $\epsilon$ in the heuristic-based dimensionality reduction approach. Y-axis on the left represents the number of distinct passenger flows in $\mathcal{F}$, i.e., $|\mathcal{F}|$. Y-axis on the right indicates the percentage of unsatisfied passengers.}
    \label{fig:heuristic_sensitivity_analyses}
\end{figure}

Sensitivity analyses of the heuristic parameter $\epsilon$ are shown in Figure \ref{fig:heuristic_sensitivity_analyses}. When the value of $\epsilon$ increases, the number of passenger flows has an exponential decrease. With fewer passenger flows considered, the robust counterpart introduces fewer constraints and variables, therefore, robust transit frequency setting problems are easier to solve. On the other hand, fewer passenger flows lead to more unsatisfied passengers with the optimized transit schedule. Regarding the percentage of unsatisfied passengers for the optimized schedule, it indicates that some passengers are not able to board a transit vehicle which is departed from the terminal station during the studied time period. In practice, unsatisfied passengers suffer longer wait times as they can board vehicles that depart from the terminal station later. In summary, robust transit schedules generated with a higher value of $\epsilon$ lead to excessive wait times by passengers. This sensitivity analysis echoes the Proposition \ref{prop:finite_loss_bound} where the objective loss monotonically increases when $\epsilon$ increases. 

\section{Conclusions and future work}
\label{sec:conclusion}

In this paper, two major issues are addressed when generating transit schedules: i) inherent demand uncertainties, and ii) gigantic OD matrices. To protect transit schedules against demand variations, a stochastic TFSP model and a robust TFSP model are introduced. A nominal optimization model is formulated to solve the TFSPs under a single transit line setting, and an extended model considering crowding levels on transit vehicles is proposed. To solve optimization problems efficiently given real-world transit instances, the TD approach is proposed based on the observation where transit demand matrices are sparse. We theoretically prove that the optimal objective function of the problem after TD is close to that of the original problem (i.e., the difference is bounded from above). Real-world transit lines operated by CTA are used to test the performances of transit schedules generated with proposed models compared to the current transit schedule. Both stochastic and robust transit schedules reduce wait times and in-vehicle travel times simultaneously for passengers over multiple demand scenarios. Compared to stochastic schedules, robust schedules further decrease in-vehicle travel times while increasing wait times by passengers.

The main limitation of this study is using heuristics to solve the robust TFSP model without proof of optimality. Meanwhile, the parameter controlling the size of the uncertainty set needs to be selected manually. Future studies could develop methodologies for decreasing problem sizes while maintaining a certain level of optimality loss. Data-driven approaches can be introduced to automatically select the value of uncertain parameter $\Gamma$. 

Another interesting research direction is pattern generation. Our model has the ability to select an optimal set of patterns to operate on a single transit line. However, how to generate a set of potential patterns for a single transit line can be a challenging task. Performances of different pattern generation algorithms can be evaluated through our proposed TFSP model. Meanwhile, other sources of uncertainty in transit systems can be considered when generating robust transit schedules, e.g., supply uncertainty (last-minute driver absence). Lastly, the proposed TFSP model can be extended to solve a network-level frequency setting problem with multiple transit lines.

\section{Acknowledgement}

The authors would like to thank Chicago Transit Authority (CTA) for offering data availability for this research.

\bibliographystyle{IEEEtranN}
\bibliography{IEEEabrv,reference}

\newpage
\appendix
\section{Robust Transit Schedules} \label{appendix:robust_schedules}

\begin{figure}[!h]
    \centering
    \includegraphics[width=\textwidth]{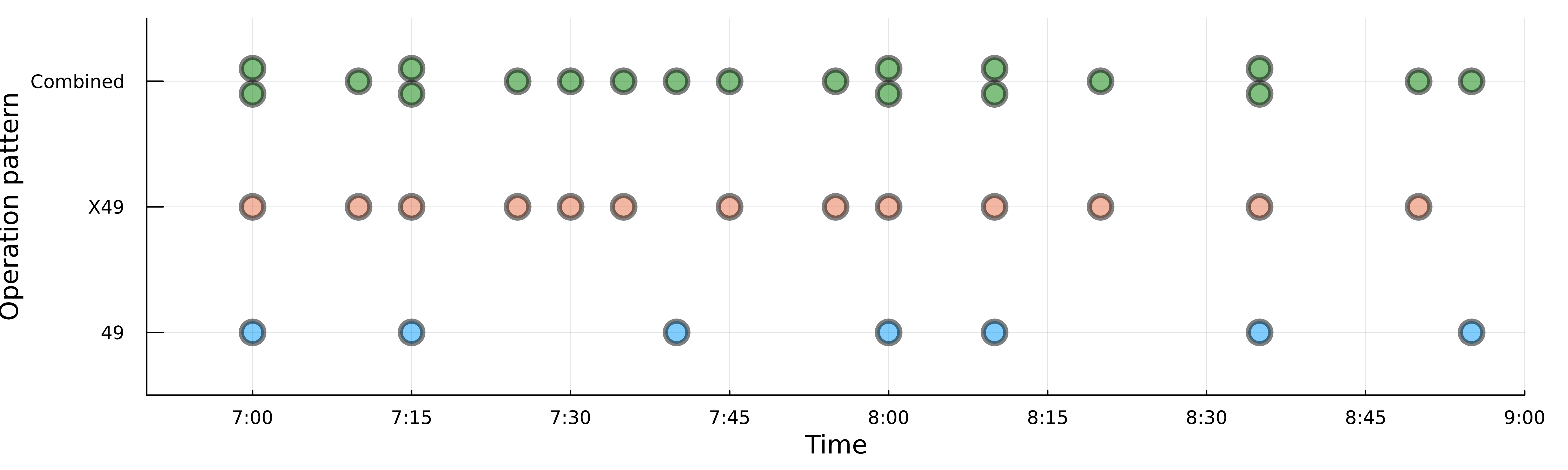}
    \caption{The robust transit schedule with $\Gamma = 0$. Each colored dot represents a departure with a specific operation pattern from the terminal stop.}
    \label{fig:robust_optimal_schedules_0}
\end{figure}

\begin{figure}[!h]
    \centering
    \includegraphics[width=\textwidth]{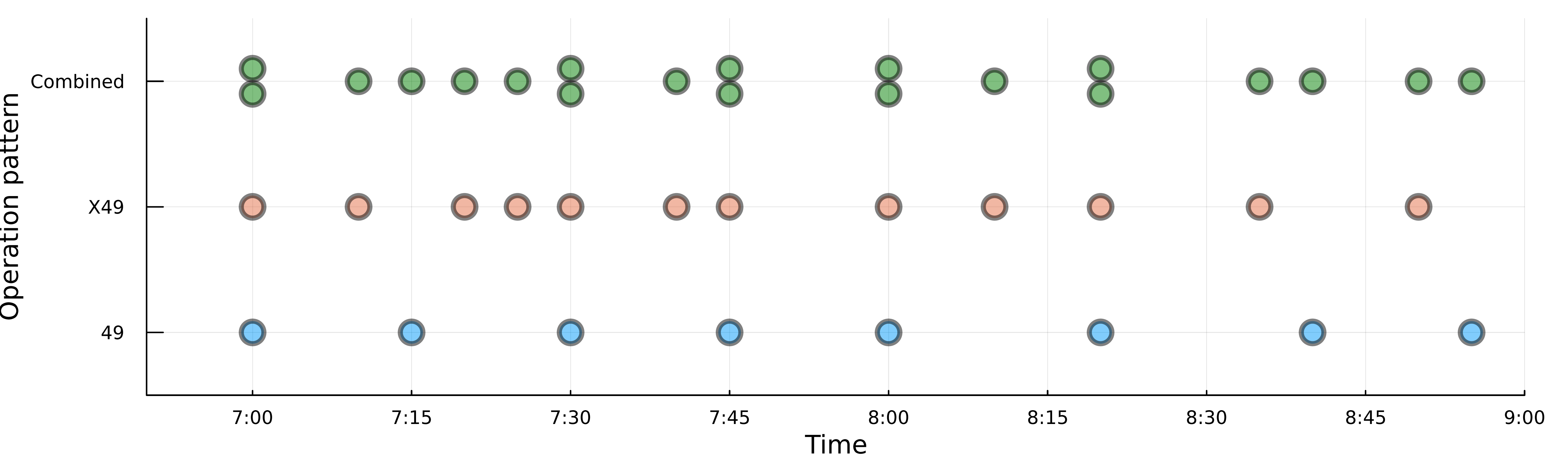}
    \caption{The robust transit schedule with $\Gamma = 1$. Each colored dot represents a departure with a specific operation pattern from the terminal stop.}
    \label{fig:robust_optimal_schedules_1}
\end{figure}

\begin{figure}[!h]
    \centering
    \includegraphics[width=\textwidth]{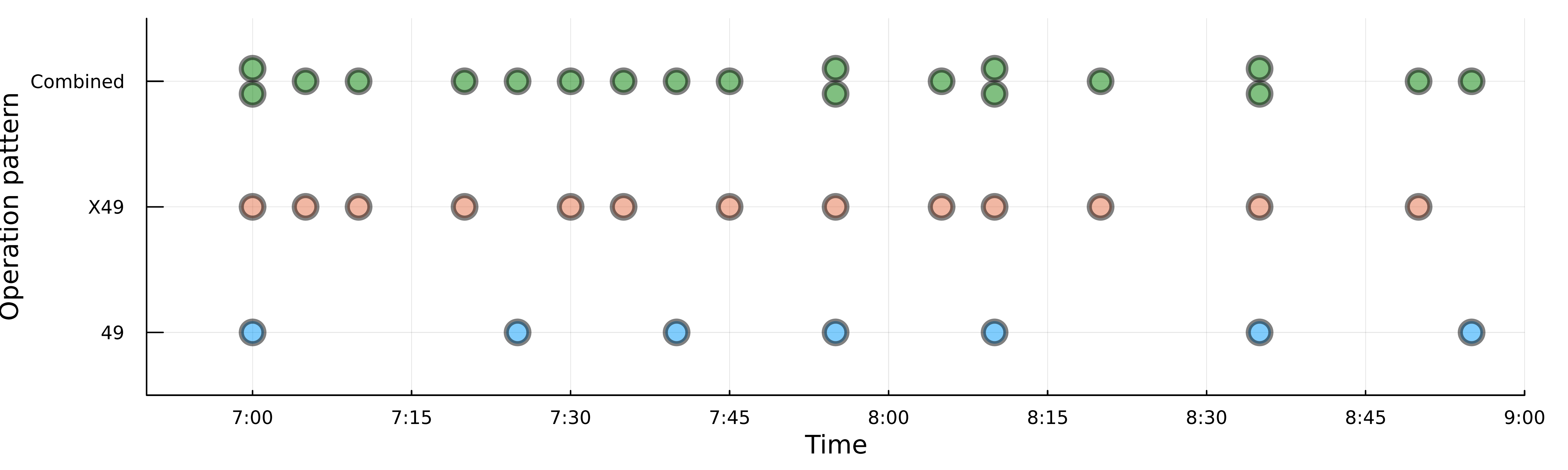}
    \caption{The robust transit schedule with $\Gamma = 2$. Each colored dot represents a departure with a specific operation pattern from the terminal stop.}
    \label{fig:robust_optimal_schedules_2}
\end{figure}

\begin{figure}[!h]
    \centering
    \includegraphics[width=\textwidth]{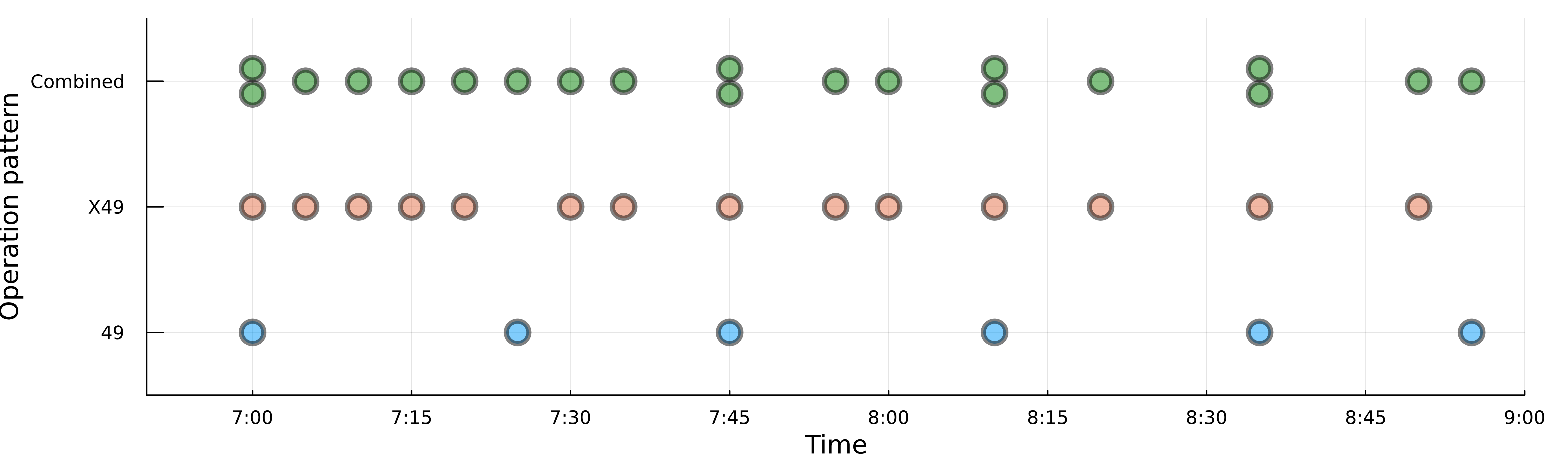}
    \caption{The robust transit schedule with $\Gamma = 3$. Each colored dot represents a departure with a specific operation pattern from the terminal stop.}
    \label{fig:robust_optimal_schedules_3}
\end{figure}

\begin{figure}[!h]
    \centering
    \includegraphics[width=\textwidth]{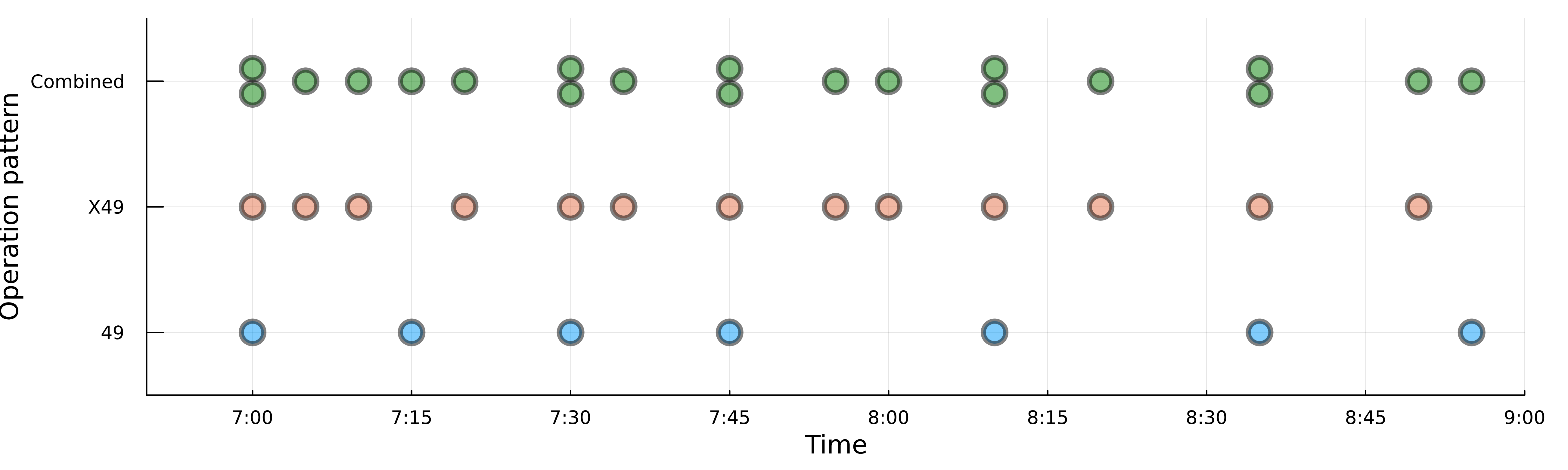}
    \caption{The robust transit schedule with $\Gamma = 4$. Each colored dot represents a departure with a specific operation pattern from the terminal stop.}
    \label{fig:robust_optimal_schedules_4}
\end{figure}

\begin{figure}[!h]
    \centering
    \includegraphics[width=\textwidth]{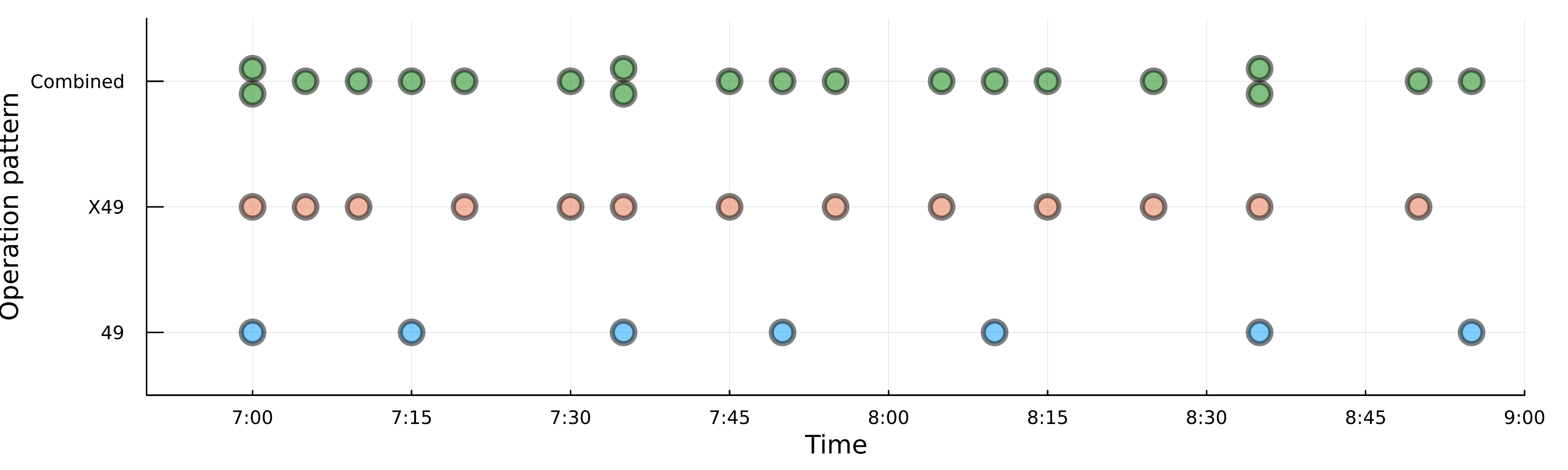}
    \caption{The robust transit schedule with $\Gamma = 5$. Each colored dot represents a departure with a specific operation pattern from the terminal stop.}
    \label{fig:robust_optimal_schedules_5}
\end{figure}

\begin{figure}[!h]
    \centering
    \includegraphics[width=\textwidth]{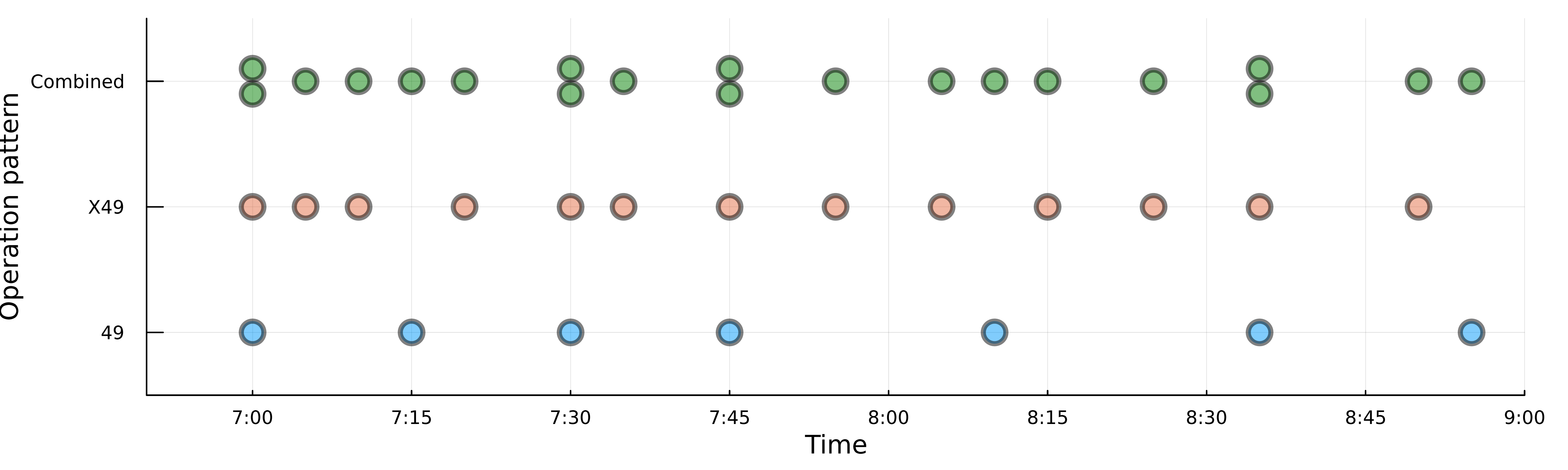}
    \caption{The robust transit schedule with $\Gamma = 6$. Each colored dot represents a departure with a specific operation pattern from the terminal stop.}
    \label{fig:robust_optimal_schedules_6}
\end{figure}

\begin{figure}[!h]
    \centering
    \includegraphics[width=\textwidth]{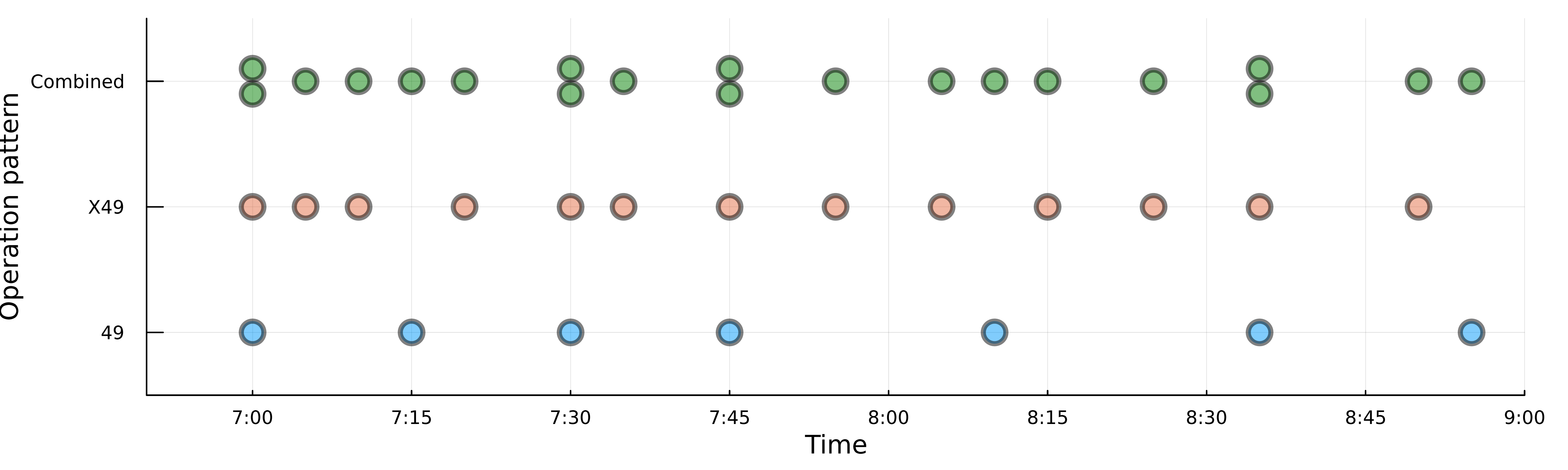}
    \caption{The robust transit schedule with $\Gamma = 7$. Each colored dot represents a departure with a specific operation pattern from the terminal stop.}
    \label{fig:robust_optimal_schedules_7}
\end{figure}

\begin{figure}[!h]
    \centering
    \includegraphics[width=\textwidth]{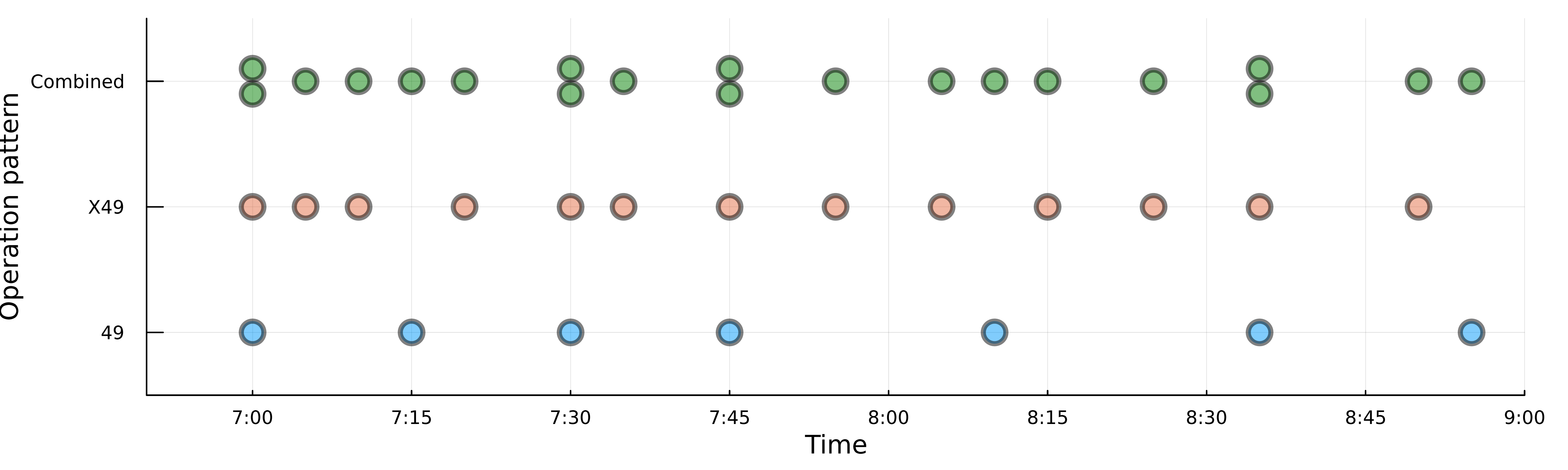}
    \caption{The robust transit schedule with $\Gamma = 8$. Each colored dot represents a departure with a specific operation pattern from the terminal stop.}
    \label{fig:robust_optimal_schedules_8}
\end{figure}

\begin{figure}[!h]
    \centering
    \includegraphics[width=\textwidth]{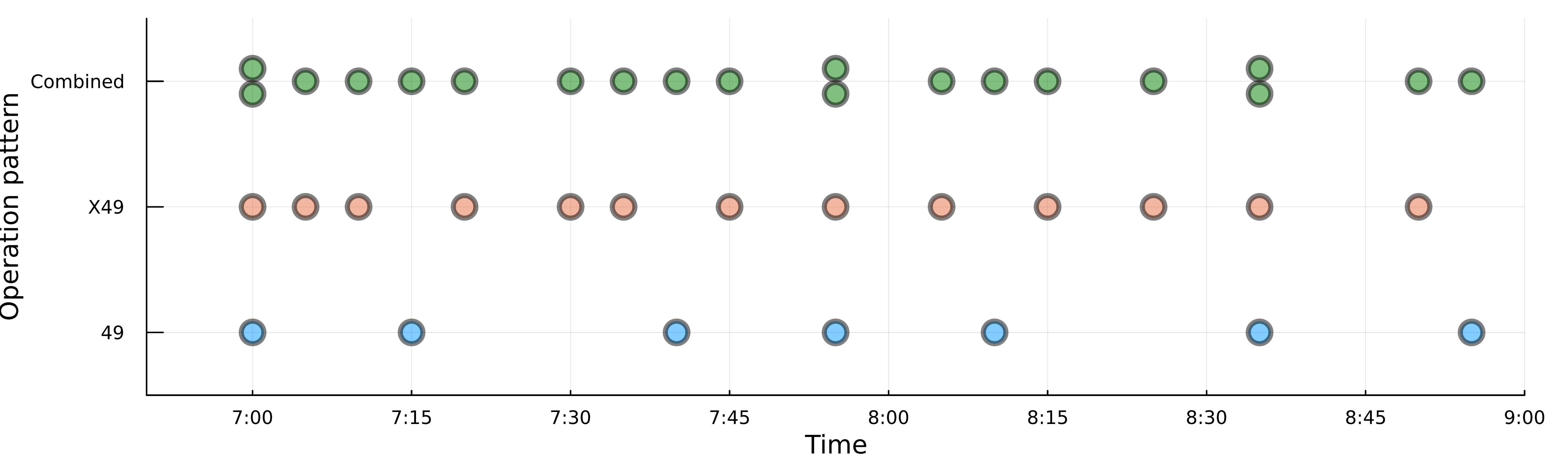}
    \caption{The robust transit schedule with $\Gamma = 9$. Each colored dot represents a departure with a specific operation pattern from the terminal stop.}
    \label{fig:robust_optimal_schedules_9}
\end{figure}

\end{document}